\documentclass{amsart}
\usepackage{graphicx}
\usepackage{multirow}
\usepackage{hyperref}
\usepackage{cleveref}

\newtheorem{theorem}{Theorem}[section]
\newtheorem{lemma}[theorem]{Lemma}
\newtheorem{proposition}{Proposition}
\newtheorem{corollary}{Corollary}

\theoremstyle{definition}

\theoremstyle{remark}
\newtheorem{remark}[theorem]{Remark}

\numberwithin{equation}{section}
\newcommand{\rM}{\mathcal{M}}
\newcommand{\cE}{\mathcal{E}}

\newcommand{\frL}{\mathfrak{L}}

\newcommand{\frl}{\mathfrak{l}}
\newcommand{\frc}{\mathfrak{c}}

\newcommand{\Yperp}{Y_{\perp}}

\newcommand{\R}{\mathbb{R}}
\newcommand{\sfg}{\mathsf{g}}
\newcommand{\bY}{Y}

\newcommand{\St}[2]{\mathrm{St}_{#1, #2}}
\newcommand{\Ric}{\textsc{Ric}}
\newcommand{\Scl}{\textsc{Scl}}
\newcommand{\cK}{\mathcal{K}}

\newcommand{\oo}{\mathfrak{o}}

\newcommand{\Herm}[1]{\mathrm{Sym}_{#1}}

\newcommand{\rmq}{\mathrm{q}}

\newcommand{\ttG}{\mathtt{G}}
\newcommand{\ttA}{\mathtt{A}}
\newcommand{\ttK}{\mathtt{K}}
\newcommand{\ttM}{\mathtt{M}}
\newcommand{\ttX}{\mathtt{X}}
\newcommand{\ttY}{\mathtt{Y}}
\newcommand{\ttZ}{\mathtt{Z}}
\newcommand{\ttV}{\mathtt{V}}
\newcommand{\frg}{\mathfrak{g}}

\newcommand{\cL}{\mathcal{L}}
\newcommand{\fra}{\mathfrak{a}}
\newcommand{\frh}{\mathfrak{h}}
\newcommand{\frm}{\mathfrak{m}}
\newcommand{\frk}{\mathfrak{k}}
\newcommand{\frb}{\mathfrak{b}}
\newcommand{\frn}{\mathfrak{n}}
\newcommand{\frv}{\mathfrak{v}}
\newcommand{\frd}{\mathfrak{d}}

\DeclareMathOperator{\lbb}{[\![}
 \DeclareMathOperator{\rbb}{]\!]}
\DeclareMathOperator{\Tr}{Tr}
\DeclareMathOperator{\dI}{I}
\DeclareMathOperator{\ft}{\mathsf{T}}
\DeclareMathOperator{\rD}{D}
\DeclareMathOperator{\rR}{R}
\DeclareMathOperator{\rP}{P}
\DeclareMathOperator{\RcM}{R^{\rM}}
\DeclareMathOperator{\hcK}{\hat{\cK}}

\DeclareMathOperator{\SOO}{SO}
\DeclareMathOperator{\ad}{ad}
\DeclareMathOperator{\addg}{ad^{\dagger}}

\begin{document}
\title[Stiefel Curvatures]{Curvatures of Stiefel manifolds with deformation metrics}
\author{Du Nguyen}

\email{nguyendu@post.harvard.edu}

\begin{abstract}We compute curvatures of a family of tractable metrics on Stiefel manifolds, introduced recently by H{\"u}per, Markina and Silva Leite, which includes the well-known embedded and canonical metrics on Stiefel manifolds as special cases. The metrics could be identified with the Cheeger deformation metrics. We identify parameter values in the family to make a Stiefel manifold an Einstein manifold and show Stiefel manifolds always carry an Einstein metric. We analyze the sectional curvature range and identify the parameter range where the manifold has non-negative sectional curvature. We provide the exact sectional curvature range when the number of columns in a Stiefel matrix is $2$, and a conjectural range for other cases. We derive the formulas from two approaches, one from a global curvature formula derived in our recent work, another using curvature formulas for left-invariant metrics. The second approach leads to curvature formulas for Cheeger deformation metrics on normal homogeneous spaces.
\end{abstract}
\maketitle
\section{Introduction}
In a recent paper \cite{NguyenCurvature}, we derived global formulas to compute the curvature of a manifold $\rM$, embedded differentiably in a Euclidean space $\cE$, with metric defined by an operator $\sfg$ from $\rM$ to the space of positive-definite operators on $\cE$. The formulas have similar forms to the classical formula for the curvature in local coordinates. While we have provided a few applications of those formulas in that paper, we would like to show the formula could be used to compute the curvatures for a family of manifolds important in both theory and application.

The purpose of this paper is to compute and analyze curvatures of a Stiefel manifold with the family of metrics defined in \cite{ExtCurveStiefel}. It turns out this family of metrics is the same family of metrics arising from the Cheeger deformation, which has been one of the main tools to construct non-negative curvature metrics\cite{Cheeger1973,GZ2000,Ziller2007}. Thus, the curvatures could be computed in two ways, one is from our formula using Christoffel functions, which is very similar to the local-coordinate formula, the other way is to use the relationship with the Cheeger deformation. In the second method, the Stiefel manifold is identified with a quotient manifold of the special orthogonal group with a left-invariant metric. Using a result of Michor \cite{Michor2007} and the Euler-Poisson-Arnold framework \cite{Arnold1966}, we compute the $(1,3)$-curvature tensor of the Cheeger deformation of a normal homogeneous space. The second approach provides independent confirmation of our curvature formulas. The first method probably requires lengthier calculation, however, it is straightforward conceptually and could be implemented symbolically.

Recall for two positive integers $p < n$, the real Stiefel manifold $\St{p}{n}$ consists of real orthogonal matrices $Y$ of size $n\times p$. If $\alpha_1$, $\alpha_0$ are two positive numbers, the metric in \cite{ExtCurveStiefel} could be reparameterized so that the inner product of two tangent vectors $\xi, \eta$ on $\St{p}{n}$ at $Y\in \St{p}{n}$ is given by $\alpha_0\Tr(\xi^{\ft}\eta) + (\alpha_1-\alpha_0)\Tr(\xi^{\ft}YY^{\ft}\eta)$. Set $\alpha = \alpha_1/\alpha_0$, and up to scaling we can take $\alpha_0= 1$. This family of metrics contains both well-known metrics on Stiefel manifolds, the embedded ($\alpha=1$, where the metric is induced from the embedding in $\R^{n\times p}$) and canonical metrics $(\alpha=\frac{1}{2})$ ($\St{p}{n}$ is normal homogeneous in this case). It will be shown in \cref{prop:stf_leftinv} that if $\SOO(n)$ is equipped with a Cheeger deformation metric with deformation parameter $2\alpha$ (reviewed in \cref{sec:deform}) from the right-multiplication action of $\SOO(p)$ embedded diagonally then $\SOO(n)/\SOO(n-p)$ with the quotient metric could be identified with $\St{p}{n}$ with the metric just described.

While a framework to compute curvatures for Cheeger deformation metrics is available, explicit formulas and detailed analysis are not yet known to the best of our knowledge (note \cite{RapcsakTamas} is an early paper dealing with the embedded metric). We provide formulas for Riemannian, Ricci, scalar, and sectional curvature for the Stiefel manifold equipped with this family of metrics. We show the sectional curvature range always contains a specific interval, which is likely to be the full curvature range for metrics in the family. The ends of the interval are piecewise smooth functions described in \cref{tab:sec_range}. In particular, except for some special cases, for the embedded metric on the Stiefel manifold, we show the curvature range contains the interval $[-\frac{1}{2}, 1]$, thus it could have negative curvatures, in contrast to the canonical metric, which has range $[0, \frac{5}{4}]$.

Specifically, $\St{2}{3}$ has positive curvature for $\alpha < \frac{2}{3}$, non-negative curvature for $\alpha=\frac{2}{3}$ and both negative and positive curvature for $\alpha > \frac{2}{3}$. With $n > 3$, the Stiefel manifold $\St{2}{n}$ has non-negative curvature for $\alpha \leq \frac{2}{3}$, and both negative and positive curvature for $\alpha > \frac{2}{3}$, and we identify the exact sectional curvature range in this case. For $p \geq 3$, we show $\St{p}{n}$ has non-negative curvature for $\alpha \leq \frac{1}{2}$ and both negative and positive curvature otherwise. This agrees with \cite{GZ2000} and we actually show the curvature range contains negative values in the indicated intervals.

We also show the Stiefel manifold always has an Einstein metric, and when $p >2$, there are two metrics in the family (up to a scaling factor) that make the Stiefel manifold an Einstein manifold.  We note this may be the same metric as in \cite{Sagle1970}.

For notations, if $n$ and $m$ are two positive integers, by $\R^{n\times m}$, we denote the space of $n\times m$ matrices in $\R$, the field of real numbers. We denote by $\oo(p)$ the space of antisymmetric matrices in $\R^{p\times p}$. The transpose of matrix or adjoint of an operator is denoted by $\ft$. Working on a manifold, say $\rM$, by $\rD_{\xi} F$, we denote the directional (Lie) derivative of a scalar\slash vector\slash operator-valued function $F$ on $\rM$ in direction $\xi$ (either a tangent vector defined at a point $x\in\rM$, or a vector field on $\rM$). If $\cE$ is a Euclidean space (inner product space with a positive-definite inner product), the space of linear operators on $\cE$ is denoted by $\frL(\cE, \cE)$. Similarly, we denote by $\frL(\cE\otimes \cE, \cE)$ the space of bilinear form on $\cE$ with value in $\cE$. For two positive integers $n$ and $p$, the Stiefel manifold $\St{p}{n}$ is the space of matrices $Y\in \R^{n\times p}$ satisfying $Y^{\ft}Y = \dI_p$.  The Frobenius norm is denoted by $\|\|_F$.

\section{Curvature formulas for embedded manifolds with metric operators}\label{sec:review}Let $\rM\subset \cE$ be a differentiable embedding, where $\cE$ is a Euclidean space with a given inner product $\langle\rangle_{\cE}$, and $\rM$ is a differentiable submanifold, and $\sfg$ is an operator-valued function from $\rM$ to $\frL(\cE, \cE)$, such that $\sfg$ is positive-definite, then $\sfg$ induces a Riemannian metric on $\rM$, where the inner product of two tangent vectors $\xi,\eta$ at a point $x\in\rM$ is defined by $\langle \xi, \sfg_x\eta\rangle_{\cE}$. Here, each tangent space $T_x\rM$ is identified with a subspace of $\cE$ thanks to the embedding, so $\xi, \eta$ are considered as elements of $\cE$, while $\sfg_x$ denotes the evaluation of the operator $\sfg$ at $x$.

We call $(\rM, \sfg, \cE)$ an embedded ambient structure. The embedding allows us to identify vector fields on $\rM$ with $\cE$-valued functions, thus we can take directional derivatives. A Christoffel function is a function $\Gamma$ from $\rM$ with value in $\frL(\cE\otimes\cE, \cE)$, the space of $\cE$-bilinear forms, such that for two vector fields $\ttX, \ttY$ on $\rM$, the Levi-Civita connection on $\rM$ is given by
$$\nabla_\ttX\ttY  = \rD_\ttX\ttY + \Gamma(\ttX, \ttY)$$
In \cite{NguyenCurvature} we proved the following curvature formulas for three tangent vectors $\xi, \eta, \phi$
\begin{equation}\label{eq:rc1a}
  \begin{gathered}
        \RcM_{\xi,\eta}\phi = -(\rD_{\xi}\Gamma)(\eta, \phi) + (\rD_{\eta}\Gamma)(\xi, \phi)-\Gamma(\xi, \Gamma(\eta, \phi)) +\Gamma(\eta, \Gamma(\xi, \phi))\\
      \RcM_{\xi,\eta}\phi = -(\rD_{\xi}\Gamma)(\eta, \phi) + (\rD_{\eta}\Gamma)(\xi, \phi)-\Gamma(\Gamma(\phi, \eta), \xi)) +\Gamma(\Gamma(\phi, \xi), \eta)
\end{gathered}
\end{equation}
where $\rD_{\xi}\Gamma$ denotes the directional derivative of $\Gamma$, considered as an operator-valued function, in the direction $\xi$, for example. The curvature for three vector fields $\ttX, \ttY, \ttZ$ is defined in the convention $$\RcM_{\ttX\ttY }\ttZ = \nabla_{[\ttX, \ttY]} \ttZ - \nabla_{\ttX} \nabla_{\ttY} \ttZ + \nabla_{\ttY} \nabla_{\ttX} \ttZ$$

\section{Curvatures of the Stiefel manifold} \label{sec:stiefel} In the following, $p < n$ are two positive integers. In \cite{ExtCurveStiefel}, the authors introduced a family of metrics on the Stiefel manifold $\St{p}{n}$ of orthogonal matrices in $\R^{n\times p}$ (thus $Y^{\ft}Y = \dI_p$). We introduced a different parameterization in \cite{Nguyen2020a}. The metric depends on two positive real numbers $\alpha_0$, $\alpha_1$ with ratio $\alpha = \frac{\alpha_1}{\alpha_0}$. In the convention of \cref{sec:review}, we have $\rM := \St{p}{n}\subset \cE :=\R^{n\times p}$, with the base inner product on $\cE$ is the Frobenius inner product, thus $\langle \omega_1, \omega_2\rangle_{\cE} = \Tr(\omega_1\omega_2^{\ft})$ for $\omega_1,\omega_2\in\cE$. Consider the metric operator  $\sfg\omega =
\sfg_Y\omega := \alpha_0\omega + (\alpha_1 - \alpha_0)YY^{\ft}\omega$, for $Y\in \St{p}{n}$, with inverse $\sfg^{-1}\omega = \alpha^{-1}_0\omega + (\alpha^{-1}_1 - \alpha^{-1}_0)YY^{\ft}\omega$ and the inner product on $\cE$ induced by $\sfg$ is $\langle\omega_1, \sfg_Y\omega_2\rangle_{\cE} = \alpha_0\Tr\omega_1\omega_2^{\ft} + (\alpha_1 - \alpha_0)\Tr\omega_1^{\ft} YY^{\ft}\omega_2$, and this induces a Riemannian metric on $\St{p}{n}$.

A geodesic equation for this metric was derived in \cite{ExtCurveStiefel}, and we provided a different derivation of a Christoffel function $\Gamma$ in \cite{Nguyen2020a}. We will give another derivation of $\Gamma$ in \cref{prop:stf_leftinv} to clarify the concepts and
keep the material reasonably independent. For an orthogonal matrix $Y\in\St{p}{n}$ and $\omega, \omega_1, \omega_2\in\R^{n\times p}$, a Christoffel function is
\begin{equation}\begin{gathered}\label{stf_gamma}
    \Gamma(\omega_1, \omega_2) = \frac{1}{2}\bY(\omega_1^{\ft}\omega_2+\omega_2^{\ft}\omega_1) +(1-\alpha)(\dI_n-\bY\bY^{\ft})(\omega_1\omega_2^{\ft}+\omega_2\omega_1^{\ft})\bY
\end{gathered}
\end{equation}
We can extend $Y$ to a full basis $(Y|\Yperp)$ of $\R^n$, by adding $\Yperp$, an orthogonal complement. Thus, $\Yperp\Yperp^{\ft} = \dI_n - YY^{\ft}$, $\Yperp^{\ft}\Yperp = \dI_{n-p}, Y^{\ft}\Yperp = 0, \Yperp^{\ft}Y = 0$. Any matrix $\omega\in\cE=\R^{n\times p}$ could be represented in this basis as $\omega = YA +\Yperp B$ with $A\in\R^{p\times p}$, $B\in\R^{p\times(n-p)}$ and $\omega$ is a tangent vector to $\St{p}{n}$ at $Y$ if and only if $A$ is antisymmetric, $A\in\oo(p)$, or equivalently $Y^{\ft}\omega +\omega^{\ft}Y = 0$.

For two tangent vectors $\xi$ and $\eta$ at a point on the manifold, denote by $\langle\rangle_{\sfg}$ and $\|\|_{\sfg}$ the inner product and the norm defined by a metric operator $\sfg$. We will denote the wedge, the sectional curvature numerator, and the sectional curvature by
\begin{equation}
  \begin{gathered}
      ||\xi\wedge\eta||_{\sfg}^2 = ||\xi||_{\sfg}^2||\eta||_{\sfg}^2 -\langle\xi, \eta\rangle_{\sfg}^2\\
      \hcK(\xi, \eta) = \langle\RcM_{\xi, \eta}\xi, \eta\rangle_{\sfg}\\
      \cK(\xi, \eta) = \frac{\hcK(\xi, \eta)}{||\xi\wedge\eta||_{\sfg}^2}\\
      \end{gathered}
    \end{equation}
\begin{theorem}\label{prop:stiefel_cur} Representing three tangent vectors $\xi, \eta, \phi\in \R^{n\times p}$ at $Y\in\St{p}{n}$ in an orthogonal basis $(Y|\Yperp)$ of $\R^n$ as $\xi= YA_1+Y_{\perp}B_1, \eta = YA_2+\Yperp B_2, \phi=YA_3+\Yperp B_3$, where $A_1, A_2, A_3\in\oo(p)$ and $B_1, B_2, B_3 \in \R^{(n-p)\times p}$. Then the Riemannian curvature tensor is $\RcM_{\xi\eta}\phi = YA_R + \Yperp B_R$ with $A_R\in\oo(p), B_R\in\R^{(n-p)\times p}$ where
  \begin{equation}\label{eq:cur_ABCA}
\begin{gathered}
  A_R = Y^{\ft}\RcM_{\xi\eta}\phi =\frac{1-2\alpha}{4} (A_{1} B_{3}^{\ft} B_{2} -  A_{2} B_{3}^{\ft} B_{1}  -
   B_{1}^{\ft} B_{3} A_{2}  +  B_{2}^{\ft} B_{3} A_{1}) +\\
   \frac{1-\alpha}{2}(A_{3} B_{1}^{\ft} B_{2} - A_{3} B_{2}^{\ft} B_{1} -  B_{1}^{\ft} B_{2}A_{3}+ B_{2}^{\ft} B_{1} A_{3}) +\\
   \frac{1}{4}([[A_{1}, A_{2}], A_{3}]
   - A_{1} B_{2}^{\ft} B_{3}  + A_{2} B_{1}^{\ft} B_{3}   + B_{3}^{\ft} B_{1} A_{2} - B_{3}^{\ft} B_{2} A_{1})
\end{gathered}
  \end{equation}
  \begin{equation}\label{eq:cur_ABCB}
    \begin{gathered}
   B_R = \Yperp^{\ft}\RcM_{\xi\eta}\phi =
\frac{2\alpha^{2}-\alpha}{2} (B_{1} A_{3} A_{2} -  B_{2} A_{3} A_{1}) +\\ (\alpha^{2}-\alpha) (B_{3} A_{1} A_{2} -  B_{3} A_{2} A_{1}) +(1 - \alpha) (B_{3} B_{1}^{\ft} B_{2} - B_{3} B_{2}^{\ft} B_{1}) +\\ \frac{\alpha-2}{2} (B_{1} B_{2}^{\ft} B_{3} -B_{2} B_{1}^{\ft} B_{3})   +
      \frac{\alpha}{2}(B_{1} A_{2} A_{3}  - B_{1} B_{3}^{\ft} B_{2} - B_{2} A_{1} A_{3} +  B_{2} B_{3}^{\ft} B_{1})
  \end{gathered}
  \end{equation}
If $p > 1$, the Ricci and scalar curvatures are given by:
  \begin{equation}\label{eq:ricci}
    \begin{gathered}
\Ric(\xi, \eta)= (\frac{2-p}{4} + (p-n)\alpha^2)\Tr(A_1A_2) + [(1-p)\alpha + (n-2)]\Tr(B_1^{\ft}B_2)
      \end{gathered}
  \end{equation}
  \begin{equation}
    \begin{gathered}
\Scl(Y) = ((1-p)\alpha + n-2)(n-p)p + ((n-p)\alpha + \frac{p-2}{4\alpha})\frac{p(p-1)}{2}
    \end{gathered}
\end{equation}    
The sectional curvature numerator $\hcK$ is  computed from one of the following
\begin{equation}\label{eq:sec_cur}
  \begin{gathered}
    \hcK = \Tr(\frac{2-3\alpha}{2}B_2^{\ft}B_1B_1^{\ft}B_2 + \frac{3\alpha-4}{2}B_2^{\ft}B_1B_2^{\ft}B_1 + B_2^{\ft}B_2B_1^{\ft}B_1-\frac{\alpha}{4}[A_1, A_2]^2) \\
+\alpha\Tr((4\alpha-3)A_1A_2B_2^{\ft}B_1 + (3-2\alpha)A_1A_2B_1^{\ft}B_2-\alpha A_2^2B_1^{\ft}B_1 -\alpha A_1^2B_2^{\ft}B_2)
  \end{gathered}
\end{equation}  
\begin{equation}\label{eq:sec_sum_sq}
  \begin{gathered}
    \hcK =  \frac{\alpha}{4}\|[A_1,A_2] +(3-4\alpha)(B_2^{\ft}B_1-B_1^{\ft}B_2 )\|_F^2 +\\
    \alpha^2\|B_1A_2-B_2A_1\|_F^2 + \frac{1}{2}\|B_1B_2^{\ft}-B_2B_1^{\ft}\|_F^2 + \frac{(1-2\alpha)^3}{2}\|B_2^{\ft}B_1-B_1^{\ft}B_2\|_F^2
\end{gathered}    
\end{equation}
In particular, if $\alpha \leq \frac{1}{2}$, the sectional curvature is non-negative. If $\xi$ and $\eta$ are orthogonal, the sectional curvature denominator is $(\alpha_1\Tr A_1A_1^{\ft} +\alpha_0\Tr B_1B_1^{\ft})(\alpha_1\Tr A_2A_2^{\ft} +\alpha_0\Tr B_2B_2^{\ft})$.
\end{theorem}
We also use the following expansion of \cref{eq:sec_sum_sq} when $A_1$ or $A_2$ is zero.
\begin{equation}\label{eq:sec_sum_sq_2}
  \begin{gathered}
    \hcK =  \frac{\alpha}{4}\|[A_1,A_2]\|_F^2 +\frac{\alpha(3-4\alpha)}{2}\Tr[A_1, A_2](B_2^{\ft}B_1-B_1^{\ft}B_2 )^{\ft} +\\
 \frac{2-3\alpha}{4}\|B_2^{\ft}B_1-B_1^{\ft}B_2\|_F^2 +  \alpha^2\|B_1A_2-B_2A_1\|_F^2 + \frac{1}{2}\|B_1B_2^{\ft}-B_2B_1^{\ft}\|_F^2
\end{gathered}    
\end{equation}

\begin{proof}
As noted, any $\omega\in \R^{n\times p}$ could be expressed as $\omega = YA +\Yperp B$, however $A$ may not be antisymmetric. By direct substitution $(\dI_n-\bY\bY^{\ft})(\eta\omega^{\ft}+\omega\eta^{\ft})\bY = \Yperp(B_2A^{\ft} - BA_2)$, hence
    $$\Gamma(\eta, \omega) = \frac{1}{2}\bY(- A_{2} A + A^{\ft} A_{2} + B^{\ft} B_{2} + B_{2}^{\ft} B) +(1-\alpha)\Yperp(B_{2} A^{\ft} - B A_{2})$$
    In particular, $Y^{\ft}\Gamma(\eta, \omega) = \frac{1}{2}(- A_{2} A + A^{\ft} A_{2} + B^{\ft} B_{2} + B_{2}^{\ft} B)$, $\Yperp^{\ft}\Gamma(\eta, \omega) = (1-\alpha)(B_2A^{\ft} - BA_2)$, and 
$$\begin{gathered}\rD_{\xi}\Gamma(\eta, \phi) =  \frac{1}{2}\xi(\eta^{\ft}\phi+\phi^{\ft}\eta) + \\(1-\alpha)\{ (\dI_n-\bY\bY^{\ft})(\eta\phi^{\ft}+\phi\eta^{\ft})\xi
-(\xi\bY^{\ft} +\bY\xi^{\ft})(\eta\phi^{\ft}+\phi\eta^{\ft})\bY 
\}\end{gathered}
  $$
Expanding $\xi, \eta,\phi$
    $$\begin{gathered}
      \Yperp^{\ft}(\rD_{\xi}\Gamma)(\eta, \phi)= \frac{1}{2}B_1(-A_2A_3-A_3A_2+B_2^{\ft}B_3+B_3^{\ft}B_2) +\\ (1-\alpha)\{B_2(-A_3 Y^{\ft}+ B_3^{\ft}\Yperp) +  B_3(-A_2Y^{\ft} + B_2^{\ft}\Yperp)\}(YA_1+\Yperp B_1) -\\ (1-\alpha)(B_1Y^{\ft}(-YA_2A_3-YA_3A_2)
      \\ =  \frac{B_{1} B_{2}^{\ft} B_{3}}{2}+ \frac{B_{1} B_{3}^{\ft} B_{2}}{2} +(\frac{1}{2} - \alpha)(B_{1} A_{2} A_{3} + B_{1} A_{3} A_{2}) +\\ (1-\alpha)(- B_{2} A_{3} A_{1}+ B_{2} B_{3}^{\ft} B_{1} - B_{3} A_{2} A_{1}+ B_{3} B_{2}^{\ft} B_{1})  
\end{gathered}$$
Simplify $Y^{\ft}(\xi Y^{\ft} + Y\xi^{\ft}) = A_1Y^{\ft}-A_1Y^{\ft}+B_1^{\ft}\Yperp^{\ft}=B_1^{\ft}\Yperp^{\ft} $
    $$\begin{gathered}
      Y^{\ft}(\rD_{\xi}\Gamma)(\eta, \phi)= \frac{1}{2}A_1(-A_2A_3+B_2^{\ft}B_3-A_3A_2+B_3^{\ft}B_2) -\\(1-\alpha)(B_1^{\ft}\Yperp^{\ft})(-\Yperp B_2A_3Y^{\ft}-\Yperp B_3A_2Y^{\ft})Y=\\
      (1- \alpha)(B_{1}^{\ft} B_{2} A_{3} + B_{1}^{\ft} B_{3} A_{2}) + \frac{1}{2}(-A_{1} A_{2} A_{3} -A_{1} A_{3} A_{2} + A_{1} B_{2}^{\ft} B_{3} + A_{1} B_{3}^{\ft} B_{2})
\end{gathered}$$
Next, use the formula for $\Gamma(\xi, \omega)$ with $\omega = \Gamma(\eta, \phi)$
    $$\begin{gathered}
    Y^{\ft}\Gamma(\xi, \Gamma(\eta, \phi)) =    
    \frac{1}{2}(-A_1(\frac{1}{2}(- A_{2} A_3 - A_3 A_{2} + B_3^{\ft} B_{2} + B_{2}^{\ft} B_3) ) +\\
    (\frac{1}{2}(- A_{2} A_3 - A_3 A_{2} + B_3^{\ft} B_{2} + B_{2}^{\ft} B_3))^{\ft}A_1 +B_1^{\ft}((1-\alpha)(-B_2A_3 - B_3A_2) ) +\\ ((1-\alpha)(-B_2A_3 - B_3A_2))^{\ft}B_1) = \\
     \frac{1-\alpha}{2}( A_{2} B_{3}^{\ft} B_{1} +  A_{3} B_{2}^{\ft} B_{1} - B_{1}^{\ft} B_{2} A_{3} - B_{1}^{\ft} B_{3} A_{2}) +
    \frac{1}{4}(A_{1} A_{2} A_{3} +\\ A_{1} A_{3} A_{2} -A_{1} B_{2}^{\ft} B_{3} - A_{1} B_{3}^{\ft} B_{2} - A_{2} A_{3} A_{1} - A_{3} A_{2} A_{1} + B_{2}^{\ft} B_{3} A_{1} + B_{3}^{\ft} B_{2} A_{1})
\end{gathered}$$
    $$\begin{gathered}\Yperp(\Gamma(\xi, \Gamma(\eta, \phi)) = (1-\alpha)\{B_1(\frac{1}{2}(- A_{2} A_3 - A_3 A_{2} + B_3^{\ft} B_{2} + B_{2}^{\ft} B_3)^{\ft} -\\ ((1-\alpha)(-B_2A_3 - B_3A_2))A_1)\}=\\
      (\alpha-1)^{2} (B_{2} A_{3} A_{1} +  B_{3} A_{2} A_{1}) + \frac{\alpha-1}{2}(B_{1} A_{2} A_{3} +  B_{1} A_{3} A_{2} -  B_{1} B_{2}^{\ft} B_{3} - B_{1} B_{3}^{\ft} B_{2})      
    \end{gathered}$$
Therefore:
    $$\begin{gathered}Y^{\ft}\RcM_{\xi\eta}\phi= -\{(1- \alpha)(B_{1}^{\ft} B_{2} A_{3} + B_{1}^{\ft} B_{3} A_{2}) +\\ \frac{1}{2}(-A_{1} A_{2} A_{3} -A_{1} A_{3} A_{2} + A_{1} B_{2}^{\ft} B_{3} + A_{1} B_{3}^{\ft} B_{2})\} +\\
      \{(1- \alpha)(B_{2}^{\ft} B_{1} A_{3} + B_{2}^{\ft} B_{3} A_{1}) + \frac{1}{2}(-A_{2} A_{1} A_{3} -A_{2} A_{3} A_{1} + A_{2} B_{1}^{\ft} B_{3} + A_{2} B_{3}^{\ft} B_{1})\}
      -\\ \{\frac{1-\alpha}{2}( A_{2} B_{3}^{\ft} B_{1} +  A_{3} B_{2}^{\ft} B_{1} - B_{1}^{\ft} B_{2} A_{3} - B_{1}^{\ft} B_{3} A_{2}) +
    \frac{1}{4}(A_{1} A_{2} A_{3} +\\ A_{1} A_{3} A_{2} -A_{1} B_{2}^{\ft} B_{3} - A_{1} B_{3}^{\ft} B_{2} - A_{2} A_{3} A_{1} - A_{3} A_{2} A_{1} + B_{2}^{\ft} B_{3} A_{1} + B_{3}^{\ft} B_{2} A_{1})\}
+ \\ \{\frac{1-\alpha}{2}( A_{1} B_{3}^{\ft} B_{2} +  A_{3} B_{1}^{\ft} B_{2} - B_{2}^{\ft} B_{1} A_{3} - B_{2}^{\ft} B_{3} A_{1}) +
\frac{1}{4}(A_{2} A_{1} A_{3} +\\ A_{2} A_{3} A_{1} -A_{2} B_{1}^{\ft} B_{3} - A_{2} B_{3}^{\ft} B_{1} - A_{1} A_{3} A_{2} - A_{3} A_{1} A_{2} + B_{1}^{\ft} B_{3} A_{2} + B_{3}^{\ft} B_{1} A_{2})\}  \\
=  \frac{1-2\alpha}{4} (A_{1} B_{3}^{\ft} B_{2} -  A_{2} B_{3}^{\ft} B_{1}  -
   B_{1}^{\ft} B_{3} A_{2}  +  B_{2}^{\ft} B_{3} A_{1}) +\\
   \frac{1-\alpha}{2}(A_{3} B_{1}^{\ft} B_{2} - A_{3} B_{2}^{\ft} B_{1} -  B_{1}^{\ft} B_{2}A_{3}+ B_{2}^{\ft} B_{1} A_{3})
   + \frac{1}{4}(A_{1} A_{2} A_{3} - A_{1} B_{2}^{\ft} B_{3}  -\\
  A_{2} A_{1} A_{3} + A_{2} B_{1}^{\ft} B_{3}  - A_{3} A_{1} A_{2} + A_{3} A_{2} A_{1} + B_{3}^{\ft} B_{1} A_{2} - B_{3}^{\ft} B_{2} A_{1})
      \end{gathered}
$$
The last expression follows from a term by term collection, for example, the coefficient of $A_1A_2A_3$ is $-(-1/2) -1/4=1/4$, and similarly for all terms with coefficient $1/4$. The coefficient for $A_1B_3^{\ft}B_2$ is $-1/2+1/4+(1-\alpha)/2=(1-2\alpha/4)$, and similar to all the terms with that coefficient.
    $$\begin{gathered}\Yperp^{\ft}\RcM_{\xi\eta}\phi = -(\frac{B_{1} B_{2}^{\ft} B_{3}}{2}+ \frac{B_{1} B_{3}^{\ft} B_{2}}{2} +(\frac{1}{2} - \alpha)(B_{1} A_{2} A_{3} + B_{1} A_{3} A_{2}) +\\ (1-\alpha)(- B_{2} A_{3} A_{1}+ B_{2} B_{3}^{\ft} B_{1} - B_{3} A_{2} A_{1}+ B_{3} B_{2}^{\ft} B_{1}) ) +\\
      (
\frac{B_{2} B_{1}^{\ft} B_{3}}{2}+ \frac{B_{2} B_{3}^{\ft} B_{1}}{2} +(\frac{1}{2} - \alpha)(B_{2} A_{1} A_{3} + B_{2} A_{3} A_{1}) +\\ (1-\alpha)(- B_{1} A_{3} A_{2}+ B_{1} B_{3}^{\ft} B_{2} - B_{3} A_{1} A_{2}+ B_{3} B_{1}^{\ft} B_{2})       
) -\\
(\alpha-1)^{2} (B_{2} A_{3} A_{1} +  B_{3} A_{2} A_{1}) - \frac{\alpha-1}{2}(B_{1} A_{2} A_{3} +  B_{1} A_{3} A_{2} -  B_{1} B_{2}^{\ft} B_{3} - B_{1} B_{3}^{\ft} B_{2})\\
+(\alpha-1)^{2} (B_{1} A_{3} A_{2} +  B_{3} A_{1} A_{2}) + \frac{\alpha-1}{2}(B_{2} A_{1} A_{3} +  B_{2} A_{3} A_{1} -  B_{2} B_{1}^{\ft} B_{3} - B_{2} B_{3}^{\ft} B_{1})  \\
\end{gathered}$$
Again, we collect term by term, (we do use a symbolic calculation program helper). The coefficient for $B_1B_2^{\ft}B_3$ is $-1/2+(\alpha-1)/2=(\alpha-2)/2$, and similar for $B_2B_1^{\ft}B3$. The coefficient for $B_1B_3^{\ft}B_2$ is $-1/2+(1-\alpha) +(\alpha-1)/2=-\alpha/2$, and similar for $B_2B_3^{\ft}B_1$. The coefficient for $B_1A_2A_3$ is $-(1/2-\alpha)-(\alpha-1)/2=\alpha/2$, and similar for  $B_2A_1A_3$. The coefficient for $B_1A_3A_2$ is $ -(\frac{1}{2}-\alpha)-(1-\alpha) -\frac{\alpha-1}{2} +(\alpha-1)^2= \alpha^2-\frac{\alpha}{2} = \frac{2\alpha^2-\alpha}{2}$ and similar for $B_2A_3A_1$. The coefficient for $B_3A_2A_1$ is $(1-\alpha)-(\alpha-1)^2 = \alpha-\alpha^2$, and $B_3A_1A_2$ follows by permutation. The coefficient for $B_3B_2^{\ft}B_1$ is $-(1-\alpha)$, and similar for $B_3B_1^{\ft}B_2$. Finally
$$\begin{gathered}
\Yperp^{\ft}\RcM_{\xi\eta}\phi=\frac{2\alpha^{2}-\alpha}{2} (B_{1} A_{3} A_{2} -  B_{2} A_{3} A_{1}) + (\alpha^{2}-\alpha) (B_{3} A_{1} A_{2} -  B_{3} A_{2} A_{1}) +\\(1 - \alpha) (B_{3} B_{1}^{\ft} B_{2} - B_{3} B_{2}^{\ft} B_{1}) +\frac{\alpha-2}{2} (B_{1} B_{2}^{\ft} B_{3} -B_{2} B_{1}^{\ft} B_{3})   +\\
      \frac{\alpha}{2}(B_{1} A_{2} A_{3}  - B_{1} B_{3}^{\ft} B_{2} - B_{2} A_{1} A_{3} +  B_{2} B_{3}^{\ft} B_{1})
      \end{gathered}
$$

The Ricci curvature is  $\Tr((A_2, B_2)\mapsto (A_R, B_R))$. Using item 3 of \cref{lem:mat_traces}, for the $A_R$ component, we compute the trace of
$$A_2\mapsto \frac{1-2\alpha}{4}(- A_{2} B_{3}^{\ft} B_{1}  - B_{1}^{\ft} B_{3} A_{2}) +  \frac{1}{4}([[A_{1}, A_{2}], A_{3}] + A_{2} B_{1}^{\ft} B_{3} + B_{3}^{\ft} B_{1} A_{2})$$
which evaluates to $\frac{1-2\alpha}{4}(p-1)\Tr(-B_3^{\ft}B_1) + \frac{1}{4}((2-p)\Tr(A_1A_3)+p\Tr(B_1^{\ft}B_3)-\Tr(B_1^{\ft}B_3))$, or $\frac{2-p}{4}\Tr(A_1A_3) + (p-1)\frac{\alpha}{2}\Tr(B_1^{\ft}B_3)$. Here, we need $p>1$, otherwise $\oo(p)$ is zero and there is no contribution from this component.

For the $B_R$ component, use item 1 of \cref{lem:mat_traces}, we compute
$$\begin{gathered}\Tr(B_2\mapsto  \frac{2\alpha^{2}-\alpha}{2} ( -  B_{2} A_{3} A_{1}) +  (1 - \alpha) (B_{3} B_{1}^{\ft} B_{2} - B_{3} B_{2}^{\ft} B_{1}) +\\
  \frac{\alpha-2}{2} (B_{1} B_{2}^{\ft} B_{3} -B_{2} B_{1}^{\ft} B_{3})   +
  \frac{\alpha}{2}(- B_{1} B_{3}^{\ft} B_{2} - B_{2} A_{1} A_{3} +  B_{2} B_{3}^{\ft} B_{1}))\\=
\frac{2\alpha^{2}-\alpha}{2}(n-p) \Tr(-A_{3} A_{1}) + (1 - \alpha)(p-1)\Tr(B_{3} B_{1}^{\ft}) +\\
  \frac{\alpha-2}{2}(1-n+p)\Tr(B_{1}^{\ft} B_{3})   +
  \frac{\alpha(n-2p)}{2}\Tr(B_{1} B_{3}^{\ft}) - \frac{\alpha(n-p)}{2}\Tr(A_{1} A_{3})
\end{gathered}$$
The Ricci curvature is
$$\begin{gathered}(\frac{2-p}{4}
-\frac{2\alpha^{2}-\alpha}{2}(n-p) - \frac{\alpha(n-p)}{2}
) \Tr(A_1A_3) +\\ \{(p-1)\frac{\alpha}{2}+(1 - \alpha)(p-1)
+\frac{\alpha-2}{2}(1-n+p) + \frac{\alpha(n-2p)}{2}\}\Tr(B_1^{\ft}B_3)\\
=  (\frac{2-p}{4} + (p-n)\alpha^2)\Tr(A_1A_2) + [(1-p)\alpha + (n-2)]\Tr(B_1^{\ft}B_2)
\end{gathered}
$$
The Ricci map is thus $(A_2, B_2)\mapsto ((\frac{p-2}{4\alpha} + (n-p)\alpha)A_2,  ((1-p)\alpha + (n-2))B_2)$, which gives us the scalar curvature formula.

For the sectional curvature, we substitute $A_1, B_1$ in place of $A_3, B_3$ in the expressions for $A_R$ and $B_R$, then compute $\Tr(-\alpha A_R A_2 + B_R B_2^{\ft})$
$$\begin{gathered}
\hcK(\xi, \eta) = \Tr(-\alpha(\frac{1-2\alpha}{4} (A_{1} B_{1}^{\ft} B_{2} -  A_{2} B_{1}^{\ft} B_{1}  -
   B_{1}^{\ft} B_{1} A_{2}  +  B_{2}^{\ft} B_{1} A_{1}) +\\
   \frac{1-\alpha}{2}(A_{1} B_{1}^{\ft} B_{2} - A_{1} B_{2}^{\ft} B_{1} -  B_{1}^{\ft} B_{2}A_{1}+ B_{2}^{\ft} B_{1} A_{1}) +\\
   \frac{1}{4}([[A_{1}, A_{2}], A_{1}]
   - A_{1} B_{2}^{\ft} B_{1}  + A_{2} B_{1}^{\ft} B_{1}   + B_{1}^{\ft} B_{1} A_{2} - B_{1}^{\ft} B_{2} A_{1}))A_{2}) \\+\Tr((
\frac{2\alpha^{2}-\alpha}{2} (B_{1} A_{1} A_{2} -  B_{2} A_{1} A_{1}) + (\alpha^{2}-\alpha) (B_{1} A_{1} A_{2} -  B_{1} A_{2} A_{1}) +\\(1 - \alpha) (B_{1} B_{1}^{\ft} B_{2} - B_{1} B_{2}^{\ft} B_{1}) +\frac{\alpha-2}{2} (B_{1} B_{2}^{\ft} B_{1} -B_{2} B_{1}^{\ft} B_{1})   +\\
\frac{\alpha}{2}(B_{1} A_{2} A_{1}  - B_{1} B_{1}^{\ft} B_{2} - B_{2} A_{1} A_{1} +  B_{2} B_{1}^{\ft} B_{1}))B_2^{\ft})
\end{gathered}
$$
We collect terms. From $-\Tr([[A_1,A_2]A_1]A_2) = \Tr[A_1, A_2][A_1, A_2]^{\ft}$, terms involving $A_1, A_2$ only are $\alpha/4\Tr[A_1, A_2][A_1, A_2]^{\ft}$.
Terms with both $A$'s and $B$'s:
$$\begin{gathered}
 \Tr(\alpha(\frac{1-2\alpha}{4} (-A_{1} B_{1}^{\ft} B_{2}A_2 +  A_{2} B_{1}^{\ft} B_{1}A_2  +
   B_{1}^{\ft} B_{1} A_{2}^2  -  B_{2}^{\ft} B_{1} A_{1}A_2) +\\
   \frac{1-\alpha}{2}(-A_{1} B_{1}^{\ft} B_{2}A_2 + A_{1} B_{2}^{\ft} B_{1}A_2 + B_{1}^{\ft} B_{2}A_{1}A_2 - B_{2}^{\ft} B_{1} A_{1}A_2) +\\
   \frac{1}{4}( A_{1} B_{2}^{\ft} B_{1}A_2  - A_{2} B_{1}^{\ft} B_{1}A_2 - B_{1}^{\ft} B_{1} A_{2}^2 + B_{1}^{\ft} B_{2} A_{1}A_{2}))) \\+\alpha\Tr(
\frac{2\alpha-1}{2} (B_{1} A_{1} A_{2}B_2^{\ft} -  B_{2} A_{1} A_{1}B_2^{\ft}) + (\alpha-1) (B_{1} A_{1} A_{2}B_2^{\ft} -  B_{1} A_{2} A_{1}B_2^{\ft}) +\\
\frac{1}{2}(B_{1} A_{2} A_{1}B_2^{\ft}   - B_{2} A_{1} A_{1}B_2^{\ft}))=\\
\alpha\Tr(
(\frac{1-\alpha}{2}+\frac{1}{4}-(\alpha-1)+\frac{1}{2})A_2A_1B_2^{\ft}B_1+
(-\frac{1-2\alpha}{4}-\frac{1-\alpha}{2})A_2A_1B_1^{\ft}B_2 +\\
(-\frac{1-2\alpha}{4}-\frac{1-\alpha}{2}+\alpha-1+\frac{2\alpha-1}{2})A_1A_2B_2^{\ft}B_1+ \\
(\frac{1-\alpha}{2}+\frac{1}{4})A_1A_2B_1^{\ft}B_2+
(2\frac{1-2\alpha}{4}-\frac{2}{4}) A_2^2B_1^{\ft}B_1+
(-\frac{2\alpha-1}{2}-\frac{1}{2})A_1^2B_2^{\ft}B_2) =\\
\alpha\Tr(
\frac{9-6\alpha}{4}A_2A_1B_2^{\ft}B_1
+\frac{4\alpha-3}{4}A_2A_1B_1^{\ft}B_2 +
\frac{12\alpha-9}{4}A_1A_2B_2^{\ft}B_1+\\ 
\frac{3-2\alpha}{4}A_1A_2B_1^{\ft}B_2
-\alpha A_2^2B_1^{\ft}B_1 -\alpha A_1^2B_2^{\ft}B_2)=\\
\alpha\Tr((4\alpha-3)A_1A_2B_2^{\ft}B_1 + (3-2\alpha)A_1A_2B_1^{\ft}B_2-\alpha A_2^2B_1^{\ft}B_1 -\alpha A_1^2B_2^{\ft}B_2)
\end{gathered}
$$
where we use $\Tr(A_2A_1B_2^{\ft}B_1) = \Tr((A_2A_1B_2^{\ft}B_1)^{\ft}) = \Tr(A_1A_2B_1^{\ft}B_2)$ and similarly $\Tr(A_2A_1B_1^{\ft}B_2) = \Tr(A_1A_2B_2^{\ft}B_1)$. Next, we collect the terms with $B_1$ and $B_2$ only:
$$\begin{gathered}
\Tr((1 - \alpha) (B_{1} B_{1}^{\ft} B_{2}B_2^{\ft} - B_{1} B_{2}^{\ft} B_{1}B_2^{\ft}) +\frac{\alpha-2}{2} (B_{1} B_{2}^{\ft} B_{1}B_2^{\ft} -B_{2} B_{1}^{\ft} B_{1}B_2^{\ft})   +\\
\frac{\alpha}{2}(  - B_{1} B_{1}^{\ft} B_{2}B_2^{\ft}  +  B_{2} B_{1}^{\ft} B_{1}B_2^{\ft})) =\\
\Tr((1-\frac{3\alpha}{2})B_1B_1^{\ft}B_2B_2^{\ft} +(\alpha-1 +\frac{\alpha-2}{2})B_1B_2^{\ft}B_1B_2^{\ft} +(-\frac{\alpha-2}{2} +\frac{\alpha}{2})B_2B_1^{\ft}B_1B_2^{\ft})\\
=\Tr(\frac{2-3\alpha}{2}B_1B_1^{\ft}B_2B_2^{\ft} +\frac{3\alpha-4}{2}B_1B_2^{\ft}B_1B_2^{\ft} +B_2B_1^{\ft}B_1B_2^{\ft})
\end{gathered}
$$
This proves \cref{eq:sec_cur}. On the other hand, it is clear on the right-hand side of \cref{eq:sec_sum_sq}, the $A$'s only term is $\frac{\alpha}{4}\Tr[A_1, A_2][A_1, A_2]^{\ft}$, the $B$'s only term is:
$$\begin{gathered}
(\frac{\alpha(3-4\alpha)^2}{4}+\frac{(1-2\alpha)^3}{2})\Tr(B_2^{\ft}B_1-B_1^{\ft}B_2)(B_2^{\ft}B_1-B_1^{\ft}B_2)^{\ft} +\\
  \frac{1}{2}\Tr(B_1B_2^{\ft}-B_2B_1^{\ft})(B_1B_2^{\ft}-B_2B_1^{\ft})^{\ft} =\\
  \frac{2-3\alpha}{4}\Tr(B_2^{\ft}B_1B_1^{\ft}B_2 -B_2^{\ft}B_1B_2^{\ft}B_1 - B_1^{\ft}B_2B_1^{\ft}B_2 +B_1^{\ft}B_2B_2^{\ft}B_1) +\\
  \frac{1}{2}\Tr(B_1B_2^{\ft}B_2B_1^{\ft} -B_1B_2^{\ft}B_1B_2^{\ft} - B_2B_1^{\ft}B_2B_1^{\ft} +B_2B_1^{\ft}B_1B_2^{\ft})\\
  =\Tr(2\frac{2-3\alpha}{4}B_1B_1^{\ft}B_2B_2^{\ft} + (-2\frac{2-3\alpha}{4}-2\frac{1}{2})B_1B_2^{\ft}B_1B_2^{\ft} + 2\frac{1}{2}B_2B_1^{\ft}B_1B_2^{\ft})\\
  =\Tr(\frac{2-3\alpha}{2}B_1B_1^{\ft}B_2B_2^{\ft} + \frac{3\alpha-4}{2}B_1B_2^{\ft}B_1B_2^{\ft} + B_2B_1^{\ft}B_1B_2^{\ft})
\end{gathered}
$$
The terms with both $A$ and $B$ in \cref{eq:sec_sum_sq} are:
$$\begin{gathered}
  \alpha\Tr(\frac{3-4\alpha}{2}(A_2A_1-A_1A_2)(B_2^{\ft}B_1-B_1^{\ft}B_2 ) + \alpha(B_1A_2-B_2A_1)^{\ft} (B_1A_2-B_2A_1))\\
  =  \alpha\Tr\{(\frac{3-4\alpha}{2}+\alpha)A_2A_1B_2^{\ft}B_1 -\frac{3-4\alpha}{2}A_2A_1B_1^{\ft}B_2 -\frac{3-4\alpha}{2}A_1A_2B_2^{\ft}B_1 +\\
  (\frac{3-4\alpha}{2}+\alpha) A_1A_2B_1^{\ft}B_2 -\alpha A_2^2B_1^{\ft}B_1 -\alpha A_1^2B_2^{\ft}B_2 \}\\
    =  \alpha\Tr\{\frac{3-2\alpha}{2}A_2A_1B_2^{\ft}B_1 -\frac{3-4\alpha}{2}A_2A_1B_1^{\ft}B_2 -\frac{3-4\alpha}{2}A_1A_2B_2^{\ft}B_1 +\\
    \frac{3-2\alpha}{2} A_1A_2B_1^{\ft}B_2 -\alpha A_2^2B_1^{\ft}B_1 -\alpha A_1^2B_2^{\ft}B_2 \}\\
    =  \alpha\Tr\{(4\alpha-3)A_1A_2B_2^{\ft}B_1 +    (3-2\alpha) A_1A_2B_1^{\ft}B_2 -\alpha A_2^2B_1^{\ft}B_1 -\alpha A_1^2B_2^{\ft}B_2 \}
\end{gathered}
$$
Therefore we have shown \cref{eq:sec_sum_sq} gives us the sectional curvature numerator.
For the sign of the sectional curvature, in \cref{eq:sec_sum_sq} the terms are all positive, except for the last, which is non-negative if $\alpha \leq\frac{1}{2}$. The formula for the curvature denominator is clear.
\end{proof}
Recall an Einstein manifold is a Riemannian manifold where the Ricci curvature tensor is proportional to the metric tensor. We have a quick application
\begin{corollary}For $p>1$, the Stiefel manifold with the metric $\sfg = \alpha_0\omega + (\alpha_1 - \alpha_0)YY^{\ft}\omega$ is an Einstein manifold if and only if $\alpha=\alpha_1/\alpha_0$ satisfies the equation
  \begin{equation}\label{eq:einstein}
    (n-1)\alpha^2 - (n-2)\alpha + \frac{p-2}{4} = 0
\end{equation}    
For $p=2$, $\alpha = \frac{n-2}{n-1}$ is the only value of $\alpha$ that makes $\St{2}{n}$ an Einstein manifold. If $p>2$, there are two values of $\alpha$ in the family making the Stiefel manifold an Einstein manifold.
\end{corollary}
\begin{proof}From \cref{eq:ricci}, the manifold is an Einstein manifold if and only if $(n-p)\alpha^2 +(p-2)/4 = \alpha (n-2 +(1-p)\alpha)$, from here \cref{eq:einstein} follows. When $p = 2$, it is clear $\frac{n-2}{n-1}$ is the only solution. When $p>2$, \cref{eq:einstein} has positive discriminant $(n-2)^2 + (p-2)(n-1)$, and  has two positive roots.
\end{proof}  
It is noted in \cite{ExtCurveStiefel} that when $p = n-1$, $\St{n-1}{n}$ is just $\SOO(n)$. Thus, we have provided $\SOO(n)$ with Einstein metrics.

\section{Sectional curvature range}
We have seen the sectional curvature numerator $\hcK$ could be expressed as a weighted sum of squares, this allows us to estimate the sectional curvature range. If $p=1$ then the Stiefel manifold is a sphere and has constant sectional curvature. Therefore we will assume $p > 1$ below.

It is easy to establish upper and lower bounds (not tight) for the sectional curvature from \cref{eq:sec_sum_sq}. Using the triangle inequality we can bound $\cK$ from  \cref{eq:sec_sum_sq} by bounding an expression of the form
$K_1=a\|[A_1, A_2]\|^2_F + b\|B_1B_2^{\ft}-B_2B_1^{\ft}\|^2_F + c\|B_1^{\ft}B_2 - B_2^{\ft}B_1\|^2_F + d\|B_1A_2 - B_2A_1\|_F^2$ by the curvature denominator $S :=(\alpha\|A_1\|_F^2 + \|B_1\|_F^2)(\alpha\|A_2\|_F^2 + \|B_2\|_F^2)$. We use the inequality $\|[X, Z]\|_F^2 \leq \|X\|_F^2\|Z\|_F^2$, for two antisymmetric matrices in $\oo(n)$ if $n>3$ (\cite{Ge2014}, lemma 2.5 provides the explicit matrices where we have equality, see also \cite{GKRgroup}, proposition 4.2). Apply that inequality with $X = \frac{1}{\sqrt{2}}\begin{bmatrix}\sqrt{2\alpha}A_1 & -B_1^{\ft}\\B_1 & 0\end{bmatrix}$, $Z= \frac{1}{\sqrt{2}}\begin{bmatrix}\sqrt{2\alpha} A_2 & B_2^{\ft}\\B_2 & 0\end{bmatrix}$ and similar inequalities for $B_1=B_2=0$, $A_1=A_2=0$, we can bound each term of $K_1$ by $S$, thus getting a bound for $\cK$.

We will attempt to provide more refined bounds. The analysis of sectional curvature range for Stiefel manifolds is more complicated than that of symmetric spaces because of the presence of both the $A$ and $B$ components. The manifold is homogeneous, therefore the sectional curvature range is the same at any point. Let $E_{ij}$ $1\leq i\leq p$ be the elementary matrix in $\R^{p\times p}$ with the $(i,j)$ entry is $1$, and other entries $0$. Let $e_{ij}$ be the elementary matrix in $\R^{(n-p)\times p}$ $(1\leq i\leq n-p, 1\leq j\leq p)$ with the $(i, j)$ entry is $1$ and the other entries are zero.

    In \cref{tab:corners}, we show sectional curvature values of $\St{p}{n}$ at several sections (pairs of linearly independent tangent vectors), each defined by a quadruple $(A_1, B_1, A_2, B_2)$. A few of those sections come from the corresponding sections for $\SOO(n)$, in \cite{Ge2014} as cited. We have noted that $\cK$ is non-negative if $\alpha\leq \frac{1}{2}$, and \cref{tab:corners} shows a section with $\cK=\frac{2-3\alpha}{2}$, thus, if $\alpha >\frac{2}{3}$, $\cK$ always has negative values in its range. When $p=2$, we will show $\cK$ is non-negative if $\alpha\leq \frac{2}{3}$. When $p>2$, $\cK$ could be negative if $\frac{1}{2}\leq \alpha\leq \frac{2}{3}$. To see this, let $A_1 = E_{12} - E_{21}, B_1 = \gamma^{1/2} e_{11}, A_2 = E_{23}-E_{32}$, $B_2 = \gamma^{1/2} e_{13}$ for $\gamma\in\R,\gamma> 0$. Thus, $[A_1,A_2] = E_{13} - E_{31}$, $B_1A_2 = B_2A_1 = 0$, $B_1B_2^{\ft} = 0$, $B_1^{\ft}B_2 -B_2^{\ft}B_1= \gamma(E_{13}-E_{31})$. By \cref{eq:sec_sum_sq_2}, the corresponding sectional curvature is
\begin{equation}\label{eq:frc}
  \frc(\gamma) = \frac{\alpha/2 + \alpha(4\alpha-3)\gamma + (2-3\alpha)\gamma^2/2}{(2\alpha+\gamma)^2}
  \end{equation}
with $\frac{d}{d\gamma}\frc(\gamma) = \alpha((7-10\alpha)\gamma -1 -6\alpha +8\alpha^2)/(\gamma + 2\alpha)^3$, $\frc$ is minimized at
\begin{equation}\label{eq:gammamin}
  \gamma_{\min}(\alpha) = (1 + 6\alpha - 8\alpha^2)/(7-10\alpha)
\end{equation}
Substitute in, the function $\frl(\alpha) := \frc(\gamma_{\min}(\alpha))$ is slightly negative for $\alpha$ in the interval $(\frac{1}{2}, \frac{7}{10})$, which contains $\frac{2}{3}$. Note that $\alpha=\frac{7}{10}$ is a removable singularity of $\frl$, and setting $\frl(\frac{7}{10}) =\lim_{\gamma\to\infty}\frc(\gamma) =\frac{1}{2}(2-3\times \frac{7}{10}) = \frac{-1}{20}$ makes it a smooth function. This function is strictly decreasing and negative in the interval $(\frac{1}{2}, \frac{7}{10})$, with $\frl(\frac{1}{2}) = 0$ and $\frl(\frac{2}{3})$ around $-0.02$.

The curvature range contains the interval between the maximum and minimum of values in \cref{tab:corners} if the condition in the last column of the table is satisfied.  For $p=2$,
\cref{prop:p_2} determines the exact curvature range. For $p>2$, numerically, the sections in that table seem to determine the range completely. For each $\alpha$, the lower and upper bound of the curvature range, found numerically by optimizing $\cK$ over the space of all sections, the Grassmann manifold of two-dimensional subspaces of $\R^{\dim\St{p}{n}}$ is within the maximal and minimal values of the sections in the table if the condition in the last column is satisfied, as shown in \cref{fig:stiefel_curv_4_3}, \ref{fig:stiefel_curv_n_3}, \ref{fig:stiefel_curv_n_geq4}, \ref{fig:stiefel_curv_n_n_1}. There, we plot the graphs of the curvatures of the list of sections as functions of $\alpha$ for the scenarios, and also plot the results of the numerical optimization for curvature range, for a set of $30$ values of $\alpha$. The optimization is done for $n=4, p=3$, $n=5 , p \in \{3, 4\}$, $n=6, p\in\{3, 4, 5\}$, $n=10, p\in\{3, 5, 10, 9\}$, $n= 100, p\in\{10, 20\}$. The curve $ll$ in the figure is for the function $\frl$. The reason the optimized maximum is sometimes smaller than the proposed maximum, for small $\alpha$, is because the optimizer may be stuck at a local maximum.

\begin{table}
  \begin{tabular}{c c c}
    \hline
    $\cK$ & $A$ and $B$ & condition\\
    \hline
    0 & $A_1=A_2=E_{12}-E_{21},B_1= 2e_{13}, B_2=-\alpha e_{13}$ & $n\geq 4, p\geq 3$\\    
    0 & $A_1=A_2=0,B_1= e_{11}, B_2=e_{22}$ & $n\geq 4, p\leq n-2$\\
    1 & $A_1=A_2=0, B_1= e_{11}, B_2 = e_{21}$&$n\geq 4, p\leq n-2$\\
    $\frac{1}{2\alpha+1}$ & $A_1=E_{12}-E_{21}, A_2=E_{1p}-E_{p1}, B_1=-e_{1p},B_2=e_{12}$& $p\geq3$\\
    $\frac{1}{8\alpha}$ &$A_1=E_{12} - E_{21}, A_2=E_{23} - E_{32}, B_1=B_2=0$ & $p\geq 3$\\
    $\frac{1}{4\alpha}$ &$A_1=E_{12} - E_{21}+E_{p-1,p}-E_{p,p-1}$ & $p\geq 4$\\
     & $A_2=E_{1,p-1} - E_{p-1,1}-E_{2,p}+E_{p,2}, B_1=B_2=0$ & \\
    $\frac{\alpha}{2}$ & $A_1=(E_{12} - E_{21}),A_2=0, B_1= 0, B_2=e_{11}$ &\\
    $\frac{2-3\alpha}{2}$ & $A_1=0,A_2=0, B_1= e_{11}, B_2=e_{12}$ &\\
    $\frac{4-3\alpha}{2}$ & $A_1=A_2=0, B_1= e_{11}+e_{22}, B_2=e_{12}-e_{21}$ &$n\geq4, p \leq n-2$\\
    $\frl(\alpha)$ & $A_1 = E_{12}-E_{21}, A_2 = E_{23}-E_{31}$ & \\
    & $B_1=\gamma_{\min}(\alpha)^{1/2}e_{11}, B_2 = \gamma_{\min}(\alpha)^{1/2}e_{13}$ & $p\geq 3$, $\alpha < 7/10$
\end{tabular}
  \caption{Sectional curvature at representative sections. $\frl(\alpha) = \frc(\gamma_{\min}(\alpha))$, from \cref{eq:gammamin} and \cref{eq:frc}.}
  \label{tab:corners}
\end{table}
\begin{figure}
  \centering
\includegraphics[scale=0.4]{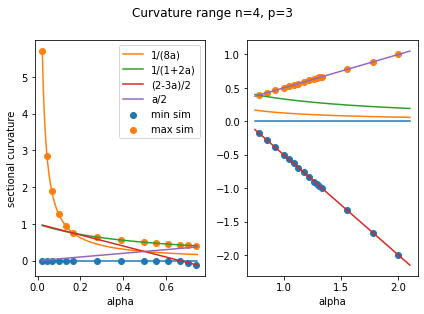}
\caption{Numerical test for curvature range $n=4, p=3$. Max, min sims are curvature ranges from numerical optimization.}
\label{fig:stiefel_curv_4_3}
\end{figure}
\begin{figure}
  \centering
\includegraphics[scale=0.4]{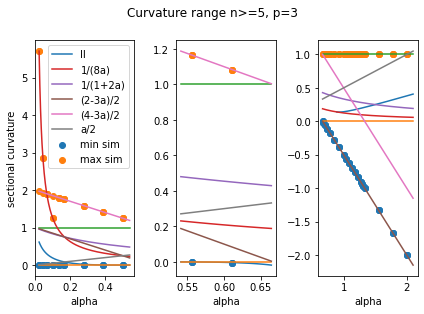}
\caption{Numerical test for curvature range $n>4, p=3$}
\label{fig:stiefel_curv_n_3}
\end{figure}

\begin{figure}
  \centering
\includegraphics[scale=0.4]{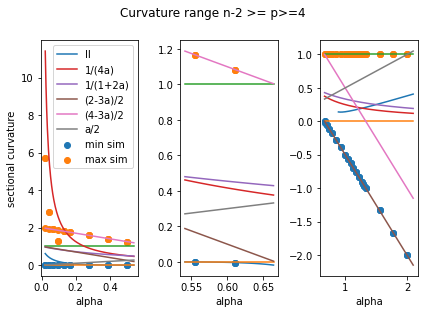}
\caption{Numerical test for curvature range $n-2\geq p\geq 4$. Max, min sims are curvature ranges from numerical optimization.}
\label{fig:stiefel_curv_n_geq4}
\end{figure}
\begin{figure}
  \centering
\includegraphics[scale=0.4]{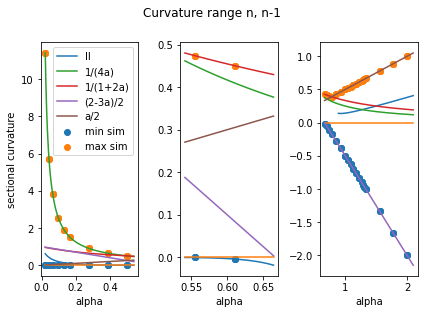}
\caption{Numerical test for curvature range $p=n-1\geq 4$}
\label{fig:stiefel_curv_n_n_1}
\end{figure}
\begin{proposition}\label{prop:p_2} If $p=2$ and $n=3$, then the sectional curvature range of $\St{p}{n}$ is $[\frac{\alpha}{2}, \frac{2-3\alpha}{2}]$ if $\alpha\leq \frac{1}{2}$ and $[\frac{2-3\alpha}{2}, \frac{\alpha}{2}]$ otherwise. In particular, if $\alpha < \frac{2}{3}$, $\St{2}{3}$ has strictly positive sectional curvature.

If $p = 2$ and $n > 3$ then the sectional curvature range is $[0, \frac{4-3\alpha}{2}]$ if $\alpha\leq \frac{2}{3}$, $[\frac{2-3\alpha}{2}, 1]$ if $\frac{2}{3}< \alpha \leq 2$ and $[\frac{2-3\alpha}{2}, \frac{\alpha}{2}]$ if $\alpha > 2$. Hence, when $n>3$, $\St{2}{n}$ has non-negative curvature if $\alpha \leq \frac{2}{3}$.
\end{proposition}
\begin{proof}When $p = 2$, $\oo(2)$ is one dimension so $[A_1, A_2] = 0$ and we can set $A_1 = (2\alpha)^{-1/2}c_1 J, A_2 = (2\alpha)^{-1/2}c_2 J$ for $J = \begin{bmatrix}0 &1 \\ -1 & 0\end{bmatrix}$, with $c_1, c_2\in \R$. Further, for two orthogonal matrices in $U, V$ of compatible dimensions, the sectional curvature is unchanged if we replace $(A_1, B_1, A_2, B_2)$ with $(VA_1V^{\ft},UB_1V, VA_2V^{\ft}, UB_2V)$. Thus, we can assume $B_1$ is rectangular diagonal, with diagonal entries denoted by $d_i$, $1\leq i\leq \min(n-p, p)$. We denote entries of $B_2$ by $b_{ij}$, $1\leq i\leq n-p, 1\leq j\leq p$. We note $B_1A_2-B_2A_1 = (2\alpha)^{-1/2}(c_2B_1J - c_1B_2J)$, and since $JJ^{\ft} = \dI_2$, $\alpha^2\|B_1A_1-B_2A_2\|_F^2 = \alpha/2(c_2^2\Tr(B_1B_1^{\ft}) + c_1^2\Tr(B_2B_2^{\ft}) - 2c_1c_2\Tr(B_1B_2^{\ft}))$. The orthogonal condition $\alpha\Tr A_1A_2^{\ft} + \Tr B_1B_2^{\ft} = 0$ implies $c_1c_2 + \Tr B_1B_2^{\ft} = c_1c_2 + \sum_{i=1}^{\min(p, n-p)} d_ib_{ii}=0$, or $c_1c_2 = -\Tr B_1B_2^{\ft}$, so $-2c_1c_2\Tr B_1B_2^{\ft} = c_1^2c_2^2+(\Tr B_1B_2)^2$. This implies
$$\alpha^2\|B_1A_1-B_2A_2\|_F^2 = \alpha/2(c_2^2\Tr(B_1B_1^{\ft}) + c_1^2\Tr(B_2B_2^{\ft}) + c_1^2c_2^2 + \Tr(B_1B_2^{\ft})^2)$$
    For the case $n=3$, from \cref{eq:sec_sum_sq_2}, the curvature numerator $\hcK$ is reduced to
    $$\frac{2-3\alpha}{2}b_{12}^2d_1^2 + \frac{\alpha}{2}(c_2^2d_1^2 + c_1^2(b_{11}^2+b_{12}^2) + c_1^2c_2^2 + d_1^2b_{11}^2)$$
    and the curvature denominator is $S=(c_1^2 + d_1^2)(c_2^2 + b_{11}^2+b_{12}^2)$. We have $\hcK - \alpha/2S = (1-2 \alpha)b_{12}^2d_1^2$, $\hcK -(1-3\alpha/2)S = (2\alpha - 1)(c_2^2d_1^2 + c_1^2(b_{11}^2+b_{12}^2) + c_1^2c_2^2 + d_1^2b_{11}^2)$. Thus, the signs of the differences are dependent on $1-2\alpha$, and $\hcK$ is between the smaller and the larger of $\alpha/2S$ and $(1-3\alpha/2)S$. The bound is tight based on \cref{tab:corners}.

    When $n > 3$, the denominator is $S=(c_1^2 + \sum_{i=1}^2 d_i^2)(c_2^2 + \sum_{ij}b^2_{ij})$. $B_1$ consists of a square diagonal block of size $2\times 2$ and the remaining zero block of size $(n-4)\times 2$. Expand $\|B_1B_2^{\ft} - B_2B_1^{\ft}\|_F^2$ by dividing $B_2$ to a square block corresponding to indices not exceeding two, which contributes $2(b_{21}d_1 - b_{12}d_2)^2$ and the remaining blocks, which contributes $2\sum_{j=1}^2\sum_{i>2}b_{ij}^2d_j^2$, $\hcK$ is
    $$\begin{gathered}
      \frac{2-3\alpha}{2}(b_{21}d_2-b_{12}d_1)^2 + \frac{\alpha}{2}(c_2^2\sum d_i^2 + c_1^2\sum_{ij}b_{ij}^2 + c_1^2c_2^2 + (\sum_{i=1}^2 d_i b_{ii})^2)\\
      + (b_{21}d_1-b_{12}d_2)^2 +\sum_{j=1}^2\sum_{i>2}b_{ij}^2d_j^2
    \end{gathered}$$
    The above expression shows when $\alpha \leq 2/3$, $\cK \geq 0$. In this case, $1\leq 2-3\alpha/2$, $\alpha/2\leq 2-3\alpha/2$, thus $\sum_{j=1}^2\sum_{i>2}b_{ij}^2d_j^2\leq (2-3\alpha/2)\sum_{i=1}^2 d_i^2\sum_{i>2}b^2_{ij}$ and
    $$\frac{\alpha}{2}(c_2^2\sum d_i^2 + c_1^2\sum_{ij}b_{ij}^2 + c_1^2c_2^2)\leq \frac{4-3\alpha}{2}(c_2^2\sum d_i^2 + c_1^2\sum_{ij}b_{ij}^2 + c_1^2c_2^2)$$

    To show $\hcK\leq (2-3\alpha/2)S$, we only need to show
    $$\frac{2-3\alpha}{2}(b_{21}d_2-b_{12}d_1)^2 + (b_{21}d_1-b_{12}d_2)^2+\frac{\alpha}{2}(\sum_{i=1}^2 d_i b_{ii})^2\leq (2-\frac{3\alpha}{2})\sum_{k=1}^2d_k^2\sum_{i\leq2}b_{ij}^2 $$
This follows from Cauchy-Schwarz's theorem, applying to three different combinations on the left-hand side then sum up the inequalities, as the first two terms on the left-hand side are dominated by $((2-3\alpha)/2 +1)(d_1^2+d_2^2)(b_{21}^2+b_{12}^2)$, while the last one is dominated by $\alpha/2(d_1^2+d_2^2)(b_{11}^2 + b_{22}^2)\leq (2-3\alpha/2)(d_1^2+d_2^2)(b_{11}^2 + b_{22}^2)$.
    
Next, when $\alpha > 2/3$, by Cauchy-Schwarz, $\hcK\geq (1-3\alpha/2)(b_{21}^2+b_{12}^2)(d_1^2+d_2^2) \geq (1-3\alpha/2)S$, as $1-3\alpha/2 < 0$. When $2/3 < \alpha \leq 2$, $\alpha/2 \leq 1$, thus $\cK\leq S$, as the first term of $\hcK$ is negative, while we can use Cauchy-Schwarz on $(\sum d_ib_{ii})^2$ and $(b_{21}d_1 - b_{12}d_2)^2$ as before. Finally, for $\alpha > 2$, $\hcK\leq \alpha/2 S$, again because the first term of $\hcK$ is negative, while the remaining terms are dominated by the corresponding terms in $\alpha/2S$, using Cauchy-Schwarz if necessary. Again, the bounds are tight using \cref{tab:corners}.
\end{proof}
We note $\St{2}{3}$ is $\SOO(3)$, and could be considered as the sphere $S^3$ with antipodal points identified (via the quaternion representation, for example). From the formula for the metric, we see this is the projective version of the Berger sphere.
\begin{proposition} For $p\geq 3$, the sectional curvature range of $\St{p}{n}$ contains an interval $I = I(n, p, \alpha)$ as described in \cref{tab:sec_range}. The first row describes the applicable combination of $(n, p)$, the columns labeled $\alpha_u$ specify the range of $\alpha$ where the interval formula next to it is applicable. The interval is applicable for $\alpha$  greater than the previous $\alpha_u$ (if exists) and not exceeding the current $\alpha_u$.
  \begin{table}[!ht]
  \begin{tabular}{|c c |c c| c c| c c|}
    \hline
\multicolumn{2}{|c|}{$(n=4, p=3)$} & \multicolumn{2}{|c|}{$(n, 3),n\geq 5$} & \multicolumn{2}{|c|}{$(n, p),n-2\geq p\geq 4$} & \multicolumn{2}{|c|}{$(n, n-1), n \geq 5$}\\
\hline
$\alpha_u$ & I &$\alpha_u$ & I &$\alpha_u$ & I &$\alpha_u$ & I \\
\hline
$\frac{1}{6}$&$[0, \frac{1}{8\alpha}]$ &$\frac{4-\sqrt{13}}{6}$ &$[0, \frac{1}{8\alpha}]$ &$\frac{4-\sqrt{10}}{6}$ &$[0, \frac{1}{4\alpha}]$ &$1/2$ & $[0,\frac{1}{4\alpha}]$\\
$1/2$& $[0, \frac{1}{1+2\alpha}]$ &$1/2$ &$[0, \frac{4-3\alpha}{2})]$ &$1/2$ &$[0, \frac{4-3\alpha}{2}]$ &$\frac{7}{10}$ & $[\frl(\alpha), \frac{1}{1+2\alpha}]$ \\
$\frac{7}{10}$ &  $[\frl(\alpha), \frac{1}{1+2\alpha}]$ &$\frac{2}{3}$ &$[\frl(\alpha), \frac{4-3\alpha}{2}]$ &$\frac{2}{3}$ &$[\frl(\alpha), \frac{4-3\alpha}{2}]$ &$\frac{\sqrt{17}-1}{4}$ & $[\frac{2-3\alpha}{2},\frac{1}{1+2\alpha}]$ \\
$\frac{\sqrt{17}-1}{4}$& $[\frac{2-3\alpha}{2}, \frac{1}{1+2\alpha}]$&$\frac{7}{10}$ &$[\frl(\alpha), 1]$ &$\frac{7}{10}$ &$[\frl(\alpha), 1]$ & $\infty$&$[\frac{2-3\alpha}{2}, \frac{\alpha}{2}]$\\
$\infty$ & $[\frac{2-3\alpha}{2}, \frac{\alpha}{2}]$ &$2$ &$[\frac{2-3\alpha}{2}, 1]$  &$2$ &$[\frac{2-3\alpha}{2}, 1]$ & & \\
& &$\infty$ &$[\frac{2-3\alpha}{2},\frac{\alpha}{2}]$ &$\infty$ &$[\frac{2-3\alpha}{2}, \frac{\alpha}{2}]$ & &\\
\hline
  \end{tabular}
  \caption{Interval contained in the sectional curvature range of the Stiefel manifold $\St{p}{n}$ with metric defined by $\alpha$. $\frl(\alpha) = \frc(\gamma_{\min})$ with $\frc$ defined in \cref{eq:frc}, and $\gamma_{\min}$ in \cref{eq:gammamin}.}
  \label{tab:sec_range}
\end{table}    
\end{proposition}
To illustrate, with $(n, p) = (4, 3)$, for $\alpha\leq \frac{1}{6}$, the sectional curvature range contains the interval $[0, \frac{1}{8\alpha}]$, for $\frac{1}{6}<\alpha\leq \frac{1}{2}$, it contains the interval $[0, \frac{1}{1+2\alpha}]$, etc. In the final row, for $\alpha > \frac{\sqrt{17}-1}{4}$, it contains the interval $[\frac{2-3\alpha}{2}, \frac{\alpha}{2}]$.
\begin{proof}It is straightforward to check that for each pair $(n, p)$ in \cref{tab:sec_range}, the values indicated correspond to a quadruple $(A_1, B_1, A_2, B_2)$ in \cref{tab:corners}, which is applicable for the pair. For example, in the case $(n, p) = (4, 3)$, the only applicable values from \cref{tab:corners} are $0$ (from the first row), $\frac{1}{2\alpha+1}$, $\frac{1}{8\alpha}$, $\frl(\alpha)$ and $\frac{2-3\alpha}{2}$. To show the sectional curvature range contains $I$, it remains to verify the lower end of $I$ is not greater than the upper end, which is immediate, as $\frl(\alpha)$ is negative between $0$ and $\frac{7}{10}$, and $\frac{2-3\alpha}{2}$ is negative for $\alpha>\frac{7}{10}>\frac{2}{3}$.

  The graphs in figures \ref{fig:stiefel_curv_4_3}, \ref{fig:stiefel_curv_n_3}, \ref{fig:stiefel_curv_n_geq4}, \ref{fig:stiefel_curv_n_n_1} display the relative values of these functions. As all the functions involved are simple algebraic functions, except for $\frl$, if we can assess the contribution of $\frl$, it will be easy to check that the lower end of $I$ corresponds to the smallest value among the applicable values, and the upper to the largest of the applicable values. The function $\gamma_{\min}$ from \cref{eq:gammamin} has a root at $\alpha_s=\frac{3+\sqrt{17}}{8}$ at around $0.89$, and is negative in the interval $(\frac{7}{10},\alpha_s)$, hence $\sqrt{\gamma_{\min}}$ and $B_1, B_2$ for this section are not defined, so $\frl(\alpha)$ cannot be an extremum for $\alpha \in (\frac{7}{10},\alpha_s)$. In the interval $[\alpha_s, 2]$, $\frl$ has the approximate range of $[0.14, 0.38]$, less than $1$, and in the interval $[\alpha_s, \frac{\sqrt{17}-1}{4}]$ it is less than $\frac{1}{1+2\alpha}$. For large $\alpha$, $\gamma_{\min}$ is approximated by $0.8\alpha$, thus $\frl(\alpha)$ has an asymptote with slope $\frac{4\times 0.8 - 3\times 0.8^2/2}{2.8^2}\approx 0.286$, smaller than the slope of $\frac{\alpha}{2}$. It is also easy to graph $\frl$ in the interim to show beyond the contribution to the lower bound in $[1/2, \frac{7}{10}]$, $\frl$ has no other effect on the curvature range.
  
With that analysis, for the case $(n, p) = (4, 3)$, the only applicable values from \cref{tab:corners} are $0$ (from the first row), $\frac{1}{2\alpha+1}$, $\frac{1}{8\alpha}$, $\frl(\alpha)$ and $\frac{2-3\alpha}{2}$. If $\alpha < 1/2$, all these functions are non-negative, and thus $0$ is the smallest value among them. When $\frac{1}{2} <\alpha < \frac{7}{10}$, $\frl(\alpha)$ is negative, and in the interval $[\frac{2}{3}, \frac{7}{10}]$, $\frac{2-3\alpha}{2}$ is also negative, but $\frl(\alpha)$ is the lesser of the two, while we have discussed $\frl(\alpha)$ has no effect for $\alpha>\frac{7}{10}$. Thus, for $\alpha >\frac{7}{10}$ the upper end of $I$ is $\max(\frac{1}{1+2\alpha}, \frac{\alpha}{2})$, with the break-even point $\frac{\sqrt{17}-1}{4}$. In general, consider the upper or lower ends of $I$ as functions of $\alpha$, the values in column $\alpha_u$ corresponds to nonsmooth points of these functions or infinity.
  
For the case $n\geq 5, p = n-1$, $(0, \frac{1}{4\alpha}, \frac{1}{8\alpha}, \frac{1}{2\alpha+1}, \frl(\alpha), \frac{2-3\alpha}{2}, \frac{\alpha}{2})$ are the applicable curvature values. Again, with $\frl$ having only an effect in $[\frac{2}{3}, \frac{7}{10}]$, it is straightforward to verify the piecewise smooth function $\max(0, \frac{1}{4\alpha},\frac{1}{8\alpha}, \frac{1}{2\alpha+1}, \frl(\alpha),\frac{2-3\alpha}{2}, \frac{\alpha}{2})$ has the form corresponding to the upper end of $I$, and the lower end corresponding to the minimum of those functions, for $\alpha > \frac{7}{10}$.  We address the case $p\geq n-2$ similarly.
\end{proof}  

For $\alpha=\frac{1}{2}$, when $p=n-1, n\geq 4$, the range contains $[0, \frac{1}{2}]$, and it could be shown to be exactly $[0, \frac{1}{2}]$ as the manifold is isometric to $\SOO(n)$ with a bi-invariant metric. If $2\leq p\leq n-2, n\geq 5$, the range contains $[0, 2-3\alpha/2] = [0, 5/4]$, which is proved to be the exact range in \cite{Rentmee}. For $\alpha=1$, the interval is $[-1/2, 1]$. From the numerical evidence mentioned, this seems to be tight. We note for $p\geq 3$, both when $\alpha$ is large or $\alpha$ is small, the curvature range becomes large.

\section{Deformation metrics on normal homogeneous manifolds}\label{sec:deform}
For a Lie group $\ttG$, with $U\in\ttG$, we will denote by $\cL_U$ the left-multiplication map and by $d\cL_U$ its differential. As usual, $\ad_A$ denotes the operator $X\mapsto [A, X]$ on the Lie algebra $\frg$ of $\ttG$ ($A, X\in\frg$). We recall a few results on curvatures of Lie groups. 
\begin{proposition}\label{prop:curv_general}Let $\ttG$ be a connected Lie group with Lie algebra $\frg$ with a  left-invariant metric given by an inner product $\langle\rangle_{\rP}$ on $\frg$. For $A\in\frg$, let $\ad_A^{\dagger}$ be the adjoint of $\ad_A$ under $\langle\rangle_{\rP}$, that means $\addg_A$ is a linear operator on $\frg$ such that $\langle[A, A_1], A_2\rangle_{\rP} = \langle A_1, \addg_AA_2\rangle_{\rP}$. Define
  \begin{equation}[A,B]_{\rP} = [A,B] -\addg_AB -\addg_BA
\end{equation}
  Let $\nabla^{\ttG}$ be the Levi-Civita connection on $\ttG$. For two vector fields $\ttX, \ttY$ on $\ttG$, there exists $\frg$-valued functions $A(U), B(U)$, $U\in\ttG$ such that $\ttX(U) = d\cL_UA(U), \ttY(U) = d\cL_UB(U)$. We have
  \begin{equation}\label{eq:nabla_GXY}
    (\nabla^{\ttG}_{\ttX}\ttY)(U) = d\cL_U((\rD_{\ttX}B)(U) + \frac{1}{2}[A(U), B(U)]_{\rP})
  \end{equation}
  where $\rD_{\ttX}B$ is the Lie-derivative of the $\frg$-valued function $B$ by the vector field $\ttX$.
  
For $\omega_1, \omega_2, \omega_3\in\frg$, the curvature of $\ttG$ at the identity is given by
  \begin{equation}\label{eq:group_curv}
\rR^{\ttG}_{\omega_1, \omega_2}\omega_3 = \frac{1}{2}[[\omega_1, \omega_2], \omega_3]_{\rP} - \frac{1}{4}[\omega_1[\omega_2, \omega_3]_ {\rP}]_{\rP} +\frac{1}{4}[\omega_2[\omega_1, \omega_3]_{\rP}]_{\rP}
\end{equation}
Let $\frk$ be a subalgebra of $\frg$ such that $\rP$ is $\ad(\frk)$-invariant, $\langle [A, K], B \rangle_{\rP} + \langle A, [K, B] \rangle_{\rP} = 0$ for $K\in\frk, A, B\in\frg$, and $\frk$ corresponds to a closed subgroup $\ttK\subset\ttG$, such that $\ttK$ acts freely and properly on $\ttG$ by isometries under right multiplication and $\ttG/\ttK$ is a homogeneous space. If $\frg = \frk\oplus\frm$ is an orthogonal decomposition under $\langle\rangle_{\rP}$, then the horizontal lift of the curvature of $\ttM =\ttG/\ttK$ at $o$, the equivariant class containing the unit of $\ttG$, evaluated at three horizontal vectors $\omega_1, \omega_2, \omega_3\in\frm$ is
  \begin{equation}\label{eq:hom_curv}
    \begin{gathered}
    \rR^{\ttM}_{\omega_1, \omega_2}\omega_3 = (\frac{1}{2}[[\omega_1, \omega_2], \omega_3]_{\rP} - \frac{1}{4}[\omega_1[\omega_2, \omega_3]_ {\rP}]_{\rP} +\frac{1}{4}[\omega_2[\omega_1, \omega_3]_{\rP}]_{\rP}\\
    +\frac{1}{2}\addg_{\omega_3}[\omega_1, \omega_2]_{\frk} -  \frac{1}{4}\addg_{\omega_1}[\omega_2, \omega_3]_{\frk} +     \frac{1}{4}\addg_{\omega_2}[\omega_1, \omega_3]_{\frk}
    )_{\frm}
    \end{gathered}
  \end{equation}  
  Here, $\omega_{\frv}$ denotes the orthogonal projection of $\omega$ to $\frv$ for an element $\omega\in\frg$ and a subspace $\frv$ of $\frg$. Also, given two vector fields $\ttX, \ttY$ on $\ttM$, which lift to horizontal vector fields $\bar{\ttX}, \bar{\ttY}$ on $\ttG$, with $\bar{\ttX}(U) = d\cL_UA(U), \bar{\ttY}=d\cL_UB(U)$ for two $\frg$-valued functions $A(U), B(U)$ on $\ttG$ then the horizontal lift of $\nabla_{\ttX}\ttY$ is given by
  \begin{equation}\label{eq:nabla_hom}
    \overline{\nabla_{\ttX}\ttY(U)} = d\cL_U((\rD_{\bar{\ttX}}B)(U) +\frac{1}{2}[A(U), B(U)]_{\rP})_{\frm}
\end{equation}    
\end{proposition}
Note that  in general $[\quad]_{\rP}$ is not anticommutative, as the term $\addg_AB +\addg_BA$ is commutative, and we have $[A,B]_{\rP} - [B,A]_{\rP} = 2[A, B]$.
\begin{proof}First, we note for three $\frg$-valued functions A, B, C
  $$\begin{gathered}\langle [A, B]_{\rP}, C\rangle_{\rP} + \langle B, [A, C]_{\rP} \rangle_{\rP} = \langle[A, B], C \rangle_{\rP} - \langle B,[A, C] \rangle_{\rP} - \langle A,[B, C] \rangle_{\rP}\\
    + \langle B, [A, C] \rangle_{\rP}  - \langle [A, B], C \rangle_{\rP} - \langle [C, B], A \rangle_{\rP} =0
  \end{gathered}$$
  For each smooth function $F:\ttG\to\sfg$, denote by $\cL[F]$ the vector field $U\mapsto d\cL_UF(U)$. Denote by $\langle\rangle_{\ttG}$ the left-invariant metric induced by $\rP$.
  For three vector fields $\ttX, \ttY, \ttZ$ with $\ttX =  \cL[A], \ttY =  \cL[B]$ and $\ttZ =  \cL[C]$ with three smooth $\frg$-valued functions $A, B, C$, we have 
  $$\begin{gathered}\rD_{\ttX}\langle \ttY, \ttZ\rangle_{\ttG} = \rD_{\ttX}\langle B, C\rangle_{\rP} = \langle \rD_{\ttX}B, C\rangle_{\rP} + \langle B, \rD_{\ttX}C\rangle_{\rP}\\
    = \langle \cL [\rD_{\ttX}B +\frac{1}{2}[A, B]_{\rP}], \ttZ\rangle_{\ttG} +
    \langle \ttY, \cL[(\rD_{\ttX}C +\frac{1}{2}[A, C]_{\rP}]\rangle_{\ttG}
  \end{gathered}$$
as the metric is left-invariant, $\rP$ is constant on $\frg$, and apply the just proved identity. We can verify $\cL [\rD_{\ttX}B +\frac{1}{2}[A, B]_{\rP}]$ satisfies the derivative rule of a connection, and we have just proved it is metric compatible. Torsion-freeness follows from $[A,B]_{\rP} - [B,A]_{\rP} = 2[A, B]$, thus $\cL [\rD_{\ttX}B +\frac{1}{2}[A, B]_{\rP}]$ is the Levi-Civita connection.
    
Equation \ref{eq:nabla_GXY} is from \cite{Michor2007}, equation 3.3.2 (the author uses a right-invariant metric). It is related to the Euler-Poisson-Arnold equation (EPDiff), see equation (55) in Arnold's classical paper \cite{Arnold1966}. See also \cite{Milnor1976}.

Equation (\ref{eq:group_curv}) now follows directly from the definition of curvature $\nabla_{[\ttX, \ttY]}\ttZ - \nabla_{\ttX}\nabla_{\ttY}\ttZ + \nabla_{\ttY}\nabla_{\ttX}\ttZ$, applying to the invariant vector fields $\cL[\omega_i], i\in\{1, 2, 3\}$. Equation \cref{eq:hom_curv} follows from the O'Neil equation (Theorem 2, \cite{ONeil1966})
, written in $(1, 3)$ tensor form. Indeed, the O'Neil tensor of two vector fields $\cL[A], \cL[B]$ on $\ttG$ for $\frg$-valued functions $A$ and $B$ evaluated at the coset $o$ is $\frac{1}{2}[A, B]_{\frk}$ as the just proved result for covariant derivatives shows the Lie bracket $\{\cL[A], \cL[B]\} = \cL[[A, B]]$, then we use Lemma 2, \cite{ONeil1966}. By properties of adjoint and projection, the right-hand side of \cref{eq:hom_curv} is the unique vector in $\frm$ such that the O'Neil equation (equation 4, theorem 2, \cite{ONeil1966}) is satisfied. Equation (\ref{eq:nabla_hom}) follows from the result for $\ttG$ and property of horizontal lift of a connection in Riemannian submersion, e.g. lemma 7.45 in \cite{ONeil1983} (because of left-invariance, we can translate the projection to the identity).
\end{proof}
For a subspace $\frv\subset\frg$, we write $\omega_{1\frv}$ for $(\omega_1)_{\frv}$, the projection of $\omega_1$ to $\frv$ ($\omega_1\in\frg$). We write $[\omega_1, \omega_2]_{\frv}$, $[\omega_1, \omega_2]_{\rP\frv}$ for the corresponding projections of brackets.

On a Lie group with a bi-invariant metric $\langle\rangle$, we now introduce a family of left-invariant metrics called the Cheeger deformation metrics (\cite{Cheeger1973,Ziller2007,GZ2000}). The Lie algebra used in the deformation will be called $\fra$ here (it is often called $\frk$, but we use $\ttK$ for the stabilizer group. We will use the letters $\fra, \frb$ corresponding to the component $A$, $B$ of the Stiefel tangent vectors as will be seen shortly). Let $\ttA$ be a connected subgroup of $\ttG$ with Lie algebra $\fra$. With the bi-invariant metric on $\ttG$, $\ttA$ acts via right multiplication as a group of isometries on $\ttG$. Give $\ttG\times\ttA$ a bi-invariant metric corresponding to the inner product on $\frg\oplus\fra$ evaluated as $\langle g, g\rangle +r\langle a, a\rangle$ for $(g, a)\in\frg\times \fra$ with $r > 0$, we have the submersion $\ttG\times \ttA\to\ttG$ given by $(U, Q)\mapsto UQ^{-1}$ ($U\in\ttG, Q\in\ttA$). Let $\frg = \fra\oplus\frn$ be an orthogonal decomposition with respect to $\langle\rangle$. The submersion induces a new metric on $\ttG$ which is shown in \cite{Ziller2007} to be
$$\langle\omega_{\frn},\omega_{\frn} \rangle + \frac{r}{(r+1)} \langle\omega_{\fra}, \omega_{\fra}\rangle$$
for $\omega\in\frg$. Denote the Cheeger deformation metric $\rP_t$ on $\frg$ by the formula $\langle\omega_{\frn},\omega_{\frn} \rangle + t\langle\omega_{\fra}, \omega_{\fra}\rangle$ for $t> 0$. At $t=1$, it is the original metric. For $t < 1$, the metric corresponds to the submersion above with $r = t/(1-t)$, thus $\ttG$ has non-negative curvature by O'Neil's equation. For $t > 1$, the metric on $\ttG\times\ttA$ is semi-Riemannian but the corresponding metric on $\ttG$ is Riemannian. If $\frn$ contains a subalgebra $\frk$ corresponding to a closed subgroup $\ttK$ of $\ttG$, such that $\frk$ commutes with $\fra$ then $\ttG/\ttK$ could be equipped with the quotient metric induced from $\rP_t$. Hence, we will consider the situation when $\frk$ is a subalgebra of an algebra $\frh$ commuting if $\fra$. Note that $\ttG/\ttK$ with the original bi-invariant metric is called a {\it normal homogeneous space} in the literature, while $\rP_t$ is no longer bi-invariant.

\begin{proposition}\label{prop:abh_split} Assume the Lie algebra $\frg$ has a bi-invariant metric $\langle\rangle$. Let $\frh\subset \frg$ be a Lie subalgebra of $\frg$ and $\frh^{\perp}$ be the orthogonal complement of $\frh$ in $\frg$ under $\langle\rangle$, $\frg = \frh\oplus\frh^{\perp}$. Then $\frb :=[\frh, \frh^{\perp}]\subset\frh^{\perp}$, or $\frh^{\perp}$ is a $\frh$-module. Let $\frh^{\perp} = \frb\oplus \fra$ be an orthogonal decomposition under $\langle\rangle$. We can characterize $\fra$ as the subspace $\{A\in \frh^{\perp}|[A,\frh]=0\}$. Then
  \begin{equation} \frg = \frh\oplus \frb\oplus \fra\end{equation}  
  We have $[\fra, \frb]\subset\frb$, $\fra$ is a Lie subalgebra of $\frg$, $[\fra, \frh] = 0$ and $\frb$ is both a $\frh$ and $\fra$ module. The correspondence $\frh\mapsto\fra$ is involutive on the set of all subalgebras of $\frg$, that means if we apply the same procedure on $\fra$, we recover $\frh$.
\end{proposition}
\begin{proof}
  Let $X\in\frh^{\perp}$ and $A, H\in\frh$. Then $\langle [A, X],H\rangle = -\langle X,[A, H]\rangle =0$ since $\frh$ is a subalgebra of $\frg$, thus $[A, X]\in\frh^{\perp}$. Assume the $\langle\rangle$-orthogonal decomposition $\frh^{\perp} = \frb\oplus\fra$ with $\frb = [\frh, \frh^{\perp}]$. For $A\in\fra$, $\langle [A, \frh], \frh^{\perp}\rangle \subset \langle A, [\frh, \frh^{\perp}]\rangle \subset\{0\}$ and $[A, \frh]\subset\frh^{\perp}$ as $\frh^{\perp}$ is a $\frh$-module. Hence, $[A, \frh]=0$ as $\langle\rangle$ is non-degenerate on $\frh^{\perp}$. Conversely, if $A\in\frh^{\perp}$ and $[A, \frh]=0$ then $\langle A, [\frh, \frh^{\perp}]\rangle \subset \langle [A, \frh], \frh^{\perp}]\rangle\subset\{0\}$, thus $A\in\fra$. We have proved $\fra$ is characterized as the subspace of $\frh^{\perp}$ such that $[A,\frh] = 0$ for $A\in\fra$.

    Next, for $A\in\fra$, $\langle [A, \frh^{\perp}], \frh\rangle \subset \langle A,[ \frh^{\perp}, \frh]\rangle\subset\{0\}$, thus $[A, \frh^{\perp}]\subset\frh^{\perp}$. Then
    $$\langle [A, [\frh, \frh^{\perp}]], \fra\rangle \subset
      \langle [[A, \frh], \frh^{\perp}], \fra\rangle + \langle [\frh,[A, \frh^{\perp}]], \fra\rangle\subset\{0\}
      $$
      as in the middle sum, the first item is zeros because $[A, \frh] = 0$, the second is
$\langle [\frh,[A, \frh^{\perp}]], \fra\rangle\subset \langle [A, \frh^{\perp}],[\frh, \fra]\rangle\subset \{0\}$ as $[\frh, \fra] = \{0\}$. This shows $[\fra, \frb]$ is in the orthogonal complement of $\fra$ in $\frh^{\perp}$, or $[\fra, \frb]\subset\frb$.

      Now, $\langle [\fra, \fra],\frh\rangle \subset \langle \fra, [\fra,\frh]\rangle \subset \{0\}$, thus $[\fra, \fra]\subset\frh^{\perp}$. But then $\langle [\fra, \fra],\frb\rangle \subset \langle \fra, [\fra,\frb]\rangle\subset\langle \fra, \frb\rangle \subset\{0\}$, hence $[\fra, \fra]\subset\fra$, therefore $\fra$ is a subalgebra of $\frg$, and $\frb$ is an $\fra$-module. Involutiveness follows from the orthogonal decomposition $\frg = \frh\oplus\frb\oplus\fra$, and the characterization of $\fra$ by the relation $[\fra, \frh]=0$, which implies $\fra^{\perp} = \frb\oplus\frh$.
\end{proof}
\begin{proposition}\label{prop:adPt}Assume $\frg$ has a bi-invariant inner product $\langle\rangle$. Let $\rP$ be a positive-definite self-adjoint operator under the inner product $\langle\rangle$. Then under the inner product $\langle\rangle_{\rP}$ defined by $\langle A_1, A_2\rangle_{\rP}:= \langle A_1, \rP A_2\rangle$, we have $\addg_A X = -\rP^{-1}[A, \rP X]$ for $X\in\frg$, or $\addg_A = -\rP^{-1}\circ\ad_A\circ \rP$.

  Let $t$ be a positive number and $\fra, \frb, \frh$ as in \cref{prop:abh_split}. Let $\frn=\frb\oplus\frh$, thus $\frg = \fra \oplus\frn$. Define the operator $\rP = \rP_t$ by $\rP \omega = t \omega_{\fra} + \omega_{\frn}$. Then for $\omega_1, \omega_2\in\frg$
\begin{equation}\label{eq:addg_Pta}
  (\addg_{\omega_{1}}\omega_{2})_{\fra} = -[\omega_{1\fra}, \omega_{2\fra}] -1/t[\omega_{1\frn},\omega_{2\frn}]_{\fra}
  \end{equation}
  \begin{equation}\label{eq:addg_Ptn}
(\addg_{\omega_{1}}\omega_{2})_{\frn} =  - [\omega_{1\fra}, \omega_{2\frb}] + t[\omega_{2\fra}, \omega_{1\frb}] -[\omega_{\frn, 1},\omega_{2\frn}]_{\frn}
      \end{equation}
  \begin{equation}\label{eq:addg_BKP}
    [\omega_1, \omega_2]_{\rP} =  [\omega_1, \omega_2] + (1-t)([\omega_{1\fra},\omega_{2\frb}] + [\omega_{2\fra},\omega_{1\frb}])\end{equation}
Let $\frk\subset \frh$ be a Lie subalgebra of $\frh$ and $\frm = \fra\oplus\frb\oplus\frd$ where $\frh = \frk\oplus\frd$ is an orthogonal decomposition, thus $\frg = \frk\oplus\frm$. For $\omega_3 \in\frg$
\begin{equation}\label{eq:oneil_P}
  (\addg_{\omega_3}[\omega_1, \omega_2]_{\frk})_{\frm} = -[\omega_{3\frm}[\omega_1, \omega_2]_{\frk}]
      \end{equation}
\end{proposition}
\begin{proof}
 Let $A, Y, X\in \frg$. From $\ad(\frg)$ invariance of $\langle\rangle$ we have
  $$\langle[A, Y], \rP X\rangle = \langle Y, -\rP \rP^{-1}[A, \rP X]\rangle$$
 which gives us the first statement.

 For \cref{eq:addg_Pta} and \cref{eq:addg_Ptn}, we expand
 $$\begin{gathered}\addg_{\omega_1}\omega_2 = -\rP^{-1}[\omega_{1\fra} + \omega_{1\frn}, t\omega_{2\fra} + \omega_{2\frn}]\\
   =(-1/t)([t\omega_{1\fra}, \omega_{2\fra}] + [\omega_{1\frn},\omega_{2\frn}]_{\fra}) - ([\omega_{1\fra}, \omega_{2\frn}] + [\omega_{1\frn}, t\omega_{2\fra}] +[\omega_{1\frn},\omega_{2\frn}])_{\frn}
 \end{gathered}$$
 then use the fact that $[\fra, \frh] = \{0\}$. Equation \ref{eq:addg_BKP} follows from this and the definition of $[\quad]_{\rP}$, using anti-commutativity to cancel $1/t([\omega_{1\frn},\omega_{2\frn}]_{\fra} + [\omega_{2\frn},\omega_{1\frn}]_{\fra})$.

 For \cref{eq:oneil_P}, let $\omega_4\in\frm$, we have
 $$\langle \addg_{\omega_3}[\omega_1,\omega_2]_{\frk},\omega_4\rangle_{\rP} =
 \langle [\omega_1,\omega_2]_{\frk},\rP[\omega_3,\omega_4]\rangle = \langle [\omega_1,\omega_2]_{\frk},[\omega_3,\omega_4]_{\frk}\rangle
 $$
 as when we expand $\rP[\omega_3,\omega_4]$, only $[\omega_3,\omega_4]_{\frk}$ could be not orthogonal to $[\omega_1,\omega_2]_{\frk}$. From here $\langle [\omega_1,\omega_2]_{\frk},[\omega_3,\omega_4]_{\frk}\rangle = \langle [\omega_1,\omega_2]_{\frk},[\omega_3,\omega_4]\rangle = -\langle [\omega_3, [\omega_1,\omega_2]_{\frk}],\omega_4\rangle$. But $[\omega_{3\frk},[\omega_1,\omega_2]_{\frk}]$ is orthogonal to $\omega_4\in\frm$, so we are left with $-\langle [\omega_{3\frm}, [\omega_1,\omega_2]_{\frk}],\omega_4\rangle =-\langle [\omega_{3\frm}, [\omega_1,\omega_2]_{\frk}],\omega_4\rangle_{\rP}$ as $[\omega_{3\frm}, [\omega_1,\omega_2]_{\frk}]_{\fra} = 0$, because $[\omega_{3\fra}, [\omega_1,\omega_2]_{\frk}] = 0$ while the remaining term is in $\frb\oplus\frh$. By \cref{prop:abh_split} $[\omega_{3\frm}, [\omega_1,\omega_2]_{\frk}]\in\frm$ since $\frm$ is the orthogonal complement of $\frk$, this proves \cref{eq:oneil_P}.
\end{proof}
Recall $o$ is the coset containing the identity in the homogeneous manifold $\ttG/\ttK$. The expression $\rR^{[0]}$ in the following theorem is the curvature of a normal homogeneous manifold, probably not usually known in this format.
\begin{proposition}\label{prop:curve_Pt} For a Lie group $\ttG$ with Lie algebra $\frg$
and a bi-invariant metric $\langle\rangle_{\ttG}$, the curvature of the homogeneous manifold $\ttM =\ttG/\ttK$ under the metric $\rP_t$ at $o$ with $\frk\subset\frh$ are subalgebras of $\frg$, ($\frg = \fra \oplus \frb\oplus\frh = \fra\oplus\frn = \frm\oplus\frk$ as in \cref{prop:abh_split}) at $\omega_1, \omega_2, \omega_3\in\frm$ is given by
      \begin{equation}
        \rR_{\omega_1, \omega_2}\omega_3 = \rR^{[0]}_{\omega_1, \omega_2}\omega_3 +
        (1-t)\rR^{[1]}_{\omega_1, \omega_2}\omega_3 + (1-t)^2\rR^{[2]}_{\omega_1, \omega_2}\omega_3        
      \end{equation}
      \begin{equation}
        \begin{gathered}
        \rR^{[0]}_{\omega_1, \omega_2}\omega_3 := \frac{1}{4}([[\omega_1, \omega_2],\omega_3]_{\frm} + 2[[\omega_1, \omega_2]_{\frk},\omega_3] -[[\omega_2, \omega_3]_{\frk},\omega_1] + [[\omega_1, \omega_3]_{\frk},\omega_2])
        \end{gathered}
      \end{equation}      
      \begin{equation}\begin{gathered}
          \rR^{[1]}_{\omega_1, \omega_2}\omega_3 :=
        \frac{1}{2}([[\omega_1, \omega_2]_{\fra}, \omega_{3\frb}] +
             [\omega_{3\fra}, [\omega_1, \omega_2]_{\frb}])\\
        - \frac{1}{4}([\omega_1, [\omega_{2\fra}, \omega_{3\frb}] +
                [\omega_{3\fra}, \omega_{2\frb}]] +
            [\omega_{1\fra}, [\omega_2, \omega_3]_{\frb}] +
            [[\omega_2, \omega_3]_{\fra}, \omega_{1\frb}])_{\frm}\\
         + \frac{1}{4}([\omega_2, [\omega_{1\fra}, \omega_{3\frb}] +
           [\omega_{3\fra}, \omega_{1\frb}]] +
            [\omega_{2\fra}, [\omega_1, \omega_3]_{\frb}] +
            [[\omega_1, \omega_3]_{\fra}, \omega_{2\frb}])_{\frm}
            \end{gathered}
      \end{equation}
   \begin{equation}\begin{gathered}
       4\rR^{[2]}_{\omega_1, \omega_2}\omega_3 := -[\omega_{1\fra}, [\omega_{2\fra}, \omega_{3\frb}] +
              [\omega_{3\fra}, \omega_{2\frb}]]
         + [\omega_{2\fra}, [\omega_{1\fra}, \omega_{3\frb}] +
           [\omega_{3\fra}, \omega_{1\frb}]]
     \end{gathered}
   \end{equation}
\end{proposition}
\begin{proof}We apply the formulas for $[\quad]_{\rP}$, with $\frn = \frb\oplus\frh$ and $\frg = \fra\oplus\frn$
        $$[[\omega_1, \omega_2], \omega_3]_{\rP, \fra} = [[\omega_1, \omega_2], \omega_3]_{\fra}
      $$
      $$[[\omega_1, \omega_2], \omega_3]_{\rP, \frn} = [[\omega_1, \omega_2], \omega_3]_{\frn} + (1-t)([[\omega_1, \omega_2]_{\fra}, \omega_{3,\frb}] +[\omega_{3, \fra}, [\omega_1, \omega_2]_{\frb}]
      $$
$$[\omega_1[\omega_2, \omega_3]_{\rP}]_{\rP}
= [\omega_1, [\omega_2, \omega_3]] +
        (1-t)
        [\omega_1, [\omega_{2, \fra}, \omega_{3,\frb}] +
          [\omega_{3,\fra}, \omega_{2,\frb}]] +$$
$$        (1-t)([\omega_{1\fra}, [\omega_2, \omega_3]_{\frb}] +
        [[\omega_2, \omega_3]_{\fra}, \omega_{1,\frb}])
        +(1-t)^2([\omega_{1\fra},
          [\omega_{2\fra}, \omega_{3\frb}] +
          [\omega_{3\fra}, \omega_{2\frb}]])        
  $$
        We now apply \cref{eq:hom_curv}. By the Jacobi identity, the $\rR^{[0]}$ component of the first line is
        $$(\frac{1}{2}[[\omega_1, \omega_2], \omega_3] - \frac{1}{4}[\omega_1, [\omega_2, \omega_3]] + \frac{1}{4}[\omega_2, [\omega_1, \omega_3]])_{\frm}=\frac{1}{4}[[\omega_1, \omega_2], \omega_3]_{\frm}$$
while the second line has the O'Neil terms $\addg_{\omega_i}[\omega_j, \omega_k]_{\frk}$ ($i, j, k$ in a permutation of $\{1, 2, 3\}$) evaluated as $-[\omega_{i\frm}[\omega_j, \omega_k]_{\frk}]$. Since we assume $\omega_i\in\frm$, this gives us the expression for $\rR^{[0]}$. Permuting indices the collect terms, we get $\rR^{[1]}$ and $\rR^{[2]}$. Some of the expressions, for example, $\rR^{[2]}_{\omega_1\omega_2}\omega_3$ are already in $\frm$ so we do not need to apply projection again.
\end{proof}
We use these formulas to compute the Levi-Civita connection and curvature for Stiefel manifolds. For two integers $n> p$, we will describe the Stiefel manifold as $\SOO(n)/\SOO(n-p)$. Here, $\ttG = \SOO(n)$ and $\ttK = \SOO(n-p)$. We take $\frg = \oo(n)$, $\ttG=\SOO(n)$. Take the bi-invariant form to be $\frac{1}{2}\Tr(\omega_1^{\ft}\omega_2)$. We divide a matrix in $\oo(n)$ to blocks of the form $\begin{bmatrix}A & -B^{\ft}\\B & H \end{bmatrix}$, with $A\in\oo(p)$, $B\in\R^{(n-p)\times p}$ and $H\in\R^{(n-p)\times (n-p)}$, and we represent that matrix by a triple $[\![A, B, H]\!]$ to save space.

Take the subalgebra generated by the $H$ block to be $\frk=\frh=\oo(n-p)$, identified with the bottom right $(n-p)\times (n-p)$ block of $\oo(n)$, then $\frm$ is the subspace of $\oo(n)$ where the $H$-block is zero, the subalgebra $\fra$ is $\oo(p)$ identified with the $A$-block, and $\frb$ is the subspace generated by the $B$ and $B^{\ft}$-blocks, as in the below
$$\frg : \begin{bmatrix}\fra & \frb\\ \frb & \frh\end{bmatrix}\quad\quad \frn : \begin{bmatrix}0 & \frb\\ \frb & \frh\end{bmatrix}\quad\quad
    \frm : \begin{bmatrix}\fra & \frb\\ \frb & 0\end{bmatrix}$$
      The Lie and $[\quad]_{\rP}$ brackets of $\lbb A_1, B_1, H_1]\rbb, \lbb A_2, B_2, H_2\rbb\in \oo(n)$ are given by
        $$[\lbb A_1, B_1, H_1\rbb, \lbb A_2, B_2, H_2\rbb] =$$
      $$\lbb[A_1, A_2]+B_2^{\ft}B_1 -B_1^{\ft}B_2, B_1A_2 +H_1B_2 - B_2A_1 - H_2B_1, [H_1, H_2] +B_2B_1^{\ft}- B_1B_2^{\ft}\rbb$$
$$[\lbb A_1, B_1, H_1\rbb, \lbb A_2, B_2, H_2\rbb]_{\rP} =
\lbb[A_1, A_2]+B_2^{\ft}B_1 -B_1^{\ft}B_2,$$
$$  tB_1A_2 +H_1B_2 +(t-2)B_2A_1 - H_2B_1,$$
$$  [H_1, H_2] +B_2B_1^{\ft}- B_1B_2^{\ft}\rbb$$

For $U = (Y|\Yperp)\in\SOO(n)$, where $(|)$ denotes the division of a matrix in $\R^{n\times n}$ to the first $p$ (in $\R^{n\times p}$) and last $n-p$ (in $\R^{n\times (n-p)}$) column blocks, if $\omega = (\eta|\eta_{\perp})$ is a tangent vector at $U$ to $\SOO(n)$ then $\omega = d\cL_U(U^{\ft}\omega) = U\lbb Y^{\ft}\eta, \Yperp^{\ft}\eta, \Yperp^{\ft}\eta_{\perp}\rbb$.

We describe the submersion $\SOO(n)\to\St{p}{n}$, identifying $\St{p}{n}$ with $\SOO(n)/\SOO(n-p)$ by the map $U\mapsto Y$, where $U = (Y|\Yperp)$ as just described. The map is clearly a differentiable submersion on to $\St{p}{n}$, the fiber over $Y$ consists of matrices of the form $(Y|\Yperp Q)$, $Q\in\SOO(n-p)$, hence the vertical space consists of $(0|\Yperp \rmq)$, $\rmq\in\oo(n-p)$.

Equip $\SOO(n)$ with the metric $\rP_t$ in \cref{prop:adPt}. At $U=\dI_n$, the horizontal space consists of matrices of the form $\lbb A, B, 0\rbb$, with $A\in\oo(p), B\in\R^{(n-p)\times p}$, and in general, a horizontal vector is of the form $U\lbb A, B, 0\rbb$. The submersion maps $\omega = (\eta|\eta_{\perp})$ to $\eta\in\R^{n\times p}$ satisfying $Y^{\ft}\eta \in\oo(p)$.

\begin{proposition}\label{prop:stf_leftinv} With the above setting, the horizontal lift of a tangent vector $\eta\in\R^{n\times p}$ to $\St{p}{n}$ at $U=(Y|\Yperp)\in\SOO(n)$ under $\rP_t$ is $\bar{\eta} = (\eta| -Y\eta^{\ft}\Yperp)$ and the induced metric is
  \begin{equation}\langle\eta, \eta \rangle_t =
    \Tr(\eta\eta^{\ft} +(\frac{t}{2} - 1) YY^{\ft}\eta\eta^{\ft})\end{equation}
  The Levi-Civita connection for two vector fields $\ttV, \ttZ$ on $\St{p}{n}$ under this metric is given by
  \begin{equation}\label{eq:levinew}\nabla_{\ttV}\ttZ = \rD_{\ttV}\ttZ + \frac{1}{2}Y(\ttV^{\ft}\ttZ + \ttV^{\ft}\ttZ) + \frac{2-t}{2}(\dI_n-YY^{\ft})(\ttV \ttZ^{\ft} +\ttV\ttZ^{\ft})Y
\end{equation}  
The curvature $\rR_{\xi, \eta}\phi$ at $Y\in\St{p}{n}$ for three tangent vectors $\xi, \eta, \phi$ computed by \cref{prop:curv_general} is identical to that computed by \cref{eq:cur_ABCA} and (\ref{eq:cur_ABCB}) if we represent the tangent and curvature vectors in the format in \cref{prop:stiefel_cur}, and set $\alpha=t/2$.
\end{proposition}
\begin{proof} A matrix multiplication shows $U^{\ft}\bar{\eta}$ is antisymmetric and could be represented as $\lbb Y^{\ft}\eta, \Yperp^{\ft}\eta, 0\rbb\in\oo(n)$, which is horizontal at $\dI_n$, thus $\bar{\eta}$ is horizontal and maps to $\eta$, hence it is the horizontal lift.

Using the relations $\Yperp\Yperp^{\ft} + YY^{\ft} = \dI_n$ the induced metric is
$$\begin{gathered}\langle U^{\ft}\bar{\eta}, U^{\ft}\bar{\eta}\rangle_{\rP} =
  \frac{1}{2}\Tr\begin{bmatrix} Y^{\ft}\eta & -\eta^{\ft}\Yperp\\\Yperp^{\ft}\eta& 0\end{bmatrix}\begin{bmatrix} t\eta^{\ft}Y & \eta^{\ft}\Yperp\\ -\Yperp^{\ft}\eta& 0\end{bmatrix}\\
    = \frac{1}{2}\Tr (tYY^{\ft}\eta\eta^{\ft} + 2 \Yperp\Yperp^{\ft}\eta\eta^{\ft}) = \Tr(\eta\eta^{\ft} +(\frac{t}{2} - 1) YY^{\ft}\eta\eta^{\ft})\end{gathered}$$

Let $\ttV, \ttZ$ be two vector fields on $\St{p}{n}$, which lift to $\SOO(n)$-vector fields $\bar{\ttV} = (\ttV|-Y\ttV^{\ft}\Yperp)$, $\bar{\ttZ} = (\ttZ|-Y\ttZ^{\ft}\Yperp)$. Let $F = U^{\ft}\bar{\ttZ} = \lbb Y^{\ft}\ttZ, \Yperp^{\ft}\ttZ, 0\rbb$, by \cref{eq:nabla_hom}, $\nabla_{\ttV}\ttZ$ lifts to $UC_{\frm}$ with $C = \rD_{\bar{\ttV}}F + \frac{1}{2}[\lbb Y^{\ft}\ttV, \Yperp^{\ft}\ttV, 0\rbb, \lbb Y^{\ft}\ttZ, \Yperp^{\ft}\ttZ, 0\rbb]_{\rP}$. Expand the Lie-derivative and the $\rP$-bracket
$$\begin{gathered}C = \lbb\ttV^{\ft}\ttZ + Y^{\ft}\rD_{\ttV}\ttZ, -\Yperp^{\ft}\ttV Y^{\ft}\ttZ +\Yperp^{\ft}\rD_{\ttV}\ttZ, 0\rbb +\\ \frac{1}{2}\lbb[Y^{\ft}\ttV,
    Y^{\ft}\ttZ] +\ttZ^{\ft}\Yperp\Yperp^{\ft}\ttV - \ttV^{\ft}\Yperp\Yperp^{\ft}\ttZ,
  t\Yperp^{\ft}\ttV Y^{\ft}\ttZ +(t-2)\Yperp^{\ft}\ttZ Y^{\ft}\ttV, C_H\rbb
\end{gathered}$$
for $C_H\in\oo(n-p)$. Thus, the submersion maps $UC_{\frm}$ to its left $p$ columns
$$\begin{gathered} Y(\ttV^{\ft}\ttZ + Y^{\ft}\rD_{\ttV}\ttZ +\frac{1}{2}([Y^{\ft}\ttV,
  Y^{\ft}\ttZ] +\ttZ^{\ft}\Yperp\Yperp^{\ft}\ttV - \ttV^{\ft}\Yperp\Yperp^{\ft}\ttZ)) +\\
  \Yperp(-\Yperp^{\ft}\ttV Y^{\ft}\ttZ +\Yperp^{\ft}\rD_{\ttV}\ttZ +\frac{1}{2}(t\Yperp^{\ft}\ttV Y^{\ft}\ttZ +(t-2)\Yperp^{\ft}\ttZ Y^{\ft}\ttV))\\
  =\rD_{\ttV}\ttZ + Y\ttV^{\ft}\ttZ + \frac{1}{2}(YY^{\ft}\ttV Y^{\ft} \ttZ - YY^{\ft}\ttZ Y^{\ft} \ttV + Y\ttZ^{\ft}\Yperp\Yperp\ttV - Y\ttV^{\ft}\Yperp\Yperp\ttZ) \\
+\frac{1}{2} \Yperp\Yperp^{\ft}(-2\ttV Y^{\ft}\ttZ  +t\ttV Y^{\ft}\ttZ +(t-2)\ttZ Y^{\ft}\ttV)
\end{gathered}$$
The last line simplifies to
$$\frac{t-2}{2}(\dI_n-YY^{\ft})(\ttV Y^{\ft}\ttZ +\ttZ Y^{\ft}\ttV) = \frac{2-t}{2}(\dI_n-YY^{\ft})(\ttV \ttZ^{\ft} +\ttZ\ttV^{\ft})Y$$
while twice the remaining terms, except for $\rD_{\ttV}\ttZ$ is
$$\begin{gathered}2Y\ttV^{\ft}\ttZ + YY^{\ft}\ttV Y^{\ft} \ttZ - YY^{\ft}\ttZ Y^{\ft} \ttV + Y\ttZ^{\ft}(\dI_n-YY^{\ft})\ttV - Y\ttV^{\ft}(\dI_n-YY^{\ft})\ttZ \\
  = Y\ttV^{\ft}\ttZ + Y\ttZ^{\ft}\ttV +Y(Y^{\ft}\ttV +\ttV^{\ft}Y)Y^{\ft}\ttZ -Y(Y^{\ft}\ttZ+\ttZ^{\ft}Y)Y^{\ft}\ttV\\
  =Y\ttV^{\ft}\ttZ + Y\ttZ^{\ft}\ttV
\end{gathered}$$
Thus we have proved \cref{eq:levinew}.

Let us prove the curvature expressions. To show $f(t)=g(t/2)$, with $f(t) = f_0 + (1-t)f_1 + (1-t)^2f_2$ where $f_1, f_2, f_3$ are constant matrices and $g$ is a matrix-valued quadratic function in $t$, we need to show $f_0 = g(1/2)$, $-2f_1 = g'(1/2)$ and $8f_2 = g''(1/2)$. From left invariance we can take $U = \dI_n$. Thus, we need to compute $\rR^{[0]}, \rR^{[1]}, \rR^{[2]}$ and compare with values and derivatives of $g(\alpha) = \lbb A_R(\alpha), B_R(\alpha), 0\rbb$ with $A_R, B_R$ defined from \cref{eq:cur_ABCA} and (\ref{eq:cur_ABCB}) evaluated at $\alpha=1/2$.

Let $\xi = \omega_1, \eta=\omega_2, \phi=\omega_3$ with $\omega_i = \lbb A_i, B_i, 0\rbb$ we have $[\omega_{2\fra}, \omega_{3\frb}]$ is $\lbb 0, -B_3A_2, 0\rbb$, $[\omega_{1\fra}, [\omega_{2\fra}, \omega_{3\frb}]] = \lbb 0, B_3A_2A_1, 0\rbb$ and permuting the indices
$$4\rR^{[2]}_{\omega_1, \omega_2}\omega_3 = \lbb 0, -B_3A_2A_1 - B_2A_3A_1+ B_3A_1A_2 + B_1A_3A_2, 0\rbb$$
On the other hand, \cref{eq:cur_ABCA} and \ref{eq:cur_ABCB} gives $A_{R, \alpha=1/2}''=0$ and $B_{R, \alpha=1/2}''$ is
   $$   B_{R,\alpha=1/2}'' =
\frac{4}{2} ( B_{1} A_{3} A_{2} -  B_{2} A_{3} A_{1}) +\\ (2) (B_{3} A_{1} A_{2} -  B_{3} A_{2} A_{1})
   $$
which confirms $8\rR^{[2]}_{\omega_1, \omega_2}\omega_3 = g''(1/2)$. Next,
        $$[[\omega_1, \omega_2]_{\fra},\omega_{3\frb}] = \lbb 0, - B_3(([A_1, A_2]+B_2^{\ft}B_1 -B_1^{\ft}B_2), 0\rbb$$
        $$[\omega_{3\fra}, [\omega_1, \omega_2]_{\frb}]_{\fra} = \lbb 0, -(B_1A_2 - B_2A_1)A_3, 0\rbb$$
        $$[\omega_1, [\omega_{2\fra}, \omega_{3\frb}]]_{\frm} = [\lbb A_1, B_1, 0\rbb, \lbb 0, -B_3A_2, 0\rbb]_{\frm} = \lbb A_2B_3^{\ft}B_1  + B_1^{\ft}B_3A_2,   B_3A_2A_1, 0\rbb$$
By permuting indices, we evaluate the $\fra$ component of $4\rR^{[1]}_{\omega_1, \omega_2}\omega_3$ from four expressions similar to $[\omega_1, [\omega_{2\fra}, \omega_{3\frb}]]_{\fra}$ as
$$\begin{gathered} -A_2B_3^{\ft}B_1  - B_1^{\ft}B_3A_2
  -A_3B_2^{\ft}B_1  - B_1^{\ft}B_2A_3\\
  +A_1B_3^{\ft}B_2  + B_2^{\ft}B_3A_1
  +A_1B_2^{\ft}B_3  + B_3^{\ft}B_2A_1
\end{gathered}
$$
and evaluate the $\frb$ component of $4\rR^{[1]}_{\omega_1, \omega_2}\omega_3$ from the remaining items as
$$\begin{gathered}
        2(- B_3([A_1, A_2]+B_2^{\ft}B_1 -B_1^{\ft}B_2)
        -(B_1A_2 - B_2A_1)A_3)\\        
        -B_3A_2A_1 - B_2A_3A_1
        +(B_2A_3 - B_3A_2)A_1
        +B_1([A_2, A_3]+B_3^{\ft}B_2 -B_2^{\ft}B_3)\\
        +B_3A_1A_2 + B_1A_3A_2
        - (B_1A_3 - B_3A_1)A_2
         -B_2([A_1, A_3]+B_3^{\ft}B_1 -B_1^{\ft}B_3)
\end{gathered}$$
Let us collect terms. Terms starting with $B_3$ and two $A$ factors are
$$- B_3[A_1, A_2] -B_3A_2A_1 -B_3A_2A_1 +B_3A_1A_2 + B_3A_1A_2 = 0$$
Terms starting with $B_2$ and two $A$ factors:
$$2B_2A_1A_3 - B_2A_3A_1+B_2A_3A_1-B_2[A_1, A_3] = B_2A_1A_3 +B_2A_3A_1$$
Terms starting with $B_1$ and two $A$ factors:
$$-2B_1A_2A_3+B_1[A_2, A_3]   + B_1A_3A_2 - B_1A_3A_2 = -B_1A_2A_3 -B_1A_3A_2$$
Terms with $B$'s only factors
$$\begin{gathered}        - 2B_3(B_2^{\ft}B_1 -B_1^{\ft}B_2)
        +B_1(B_3^{\ft}B_2 -B_2^{\ft}B_3) -B_2(B_3^{\ft}B_1 -B_1^{\ft}B_3)
\end{gathered}$$
On the other hand, we have
$$\begin{gathered}A_{R, \alpha=1/2}' =   \frac{-2}{4} (A_{1} B_{3}^{\ft} B_{2} -  A_{2} B_{3}^{\ft} B_{1}  -
   B_{1}^{\ft} B_{3} A_{2}  +  B_{2}^{\ft} B_{3} A_{1}) +\\
   \frac{-1}{2}(A_{3} B_{1}^{\ft} B_{2} - A_{3} B_{2}^{\ft} B_{1} -  B_{1}^{\ft} B_{2}A_{3}+ B_{2}^{\ft} B_{1} A_{3})
\end{gathered}$$
   $$\begin{gathered}B_{R, \alpha=1/2}' = 
     \frac{4(1/2)-1}{2} (B_{1} A_{3} A_{2} -  B_{2} A_{3} A_{1}) +\\ (2(1/2)-1) (B_{3} A_{1} A_{2} -  B_{3} A_{2} A_{1}) - (B_{3} B_{1}^{\ft} B_{2} - B_{3} B_{2}^{\ft} B_{1}) +\\ \frac{1}{2} (B_{1} B_{2}^{\ft} B_{3} -B_{2} B_{1}^{\ft} B_{3})
     +      \frac{1}{2}(B_{1} A_{2} A_{3}  - B_{1} B_{3}^{\ft} B_{2} - B_{2} A_{1} A_{3} +  B_{2} B_{3}^{\ft} B_{1})
\end{gathered}$$
and we can confirm by inspection $-2\rR^{[1]}_{\omega_1, \omega_2}\omega_3 = g'(1/2)$. The constant term $\rR^{[0]}$ is verified similarly, which we will not show here.
\end{proof}
\begin{remark}We have shown the metric in \cref{sec:stiefel} is $\rP_t$ for $t=\frac{\alpha}{2}$. The submersion associated with the Cheeger deformation gives a sectional curvature formula for $\ttG$ with the metric $\rP_t$ in proposition 2.4 of \cite{GZ2000}. Using the O'Neil equation and \cref{eq:oneil_P}, it implies the following sectional curvature formula for $\ttM=\ttG/\ttK$ (the norm $\|\|$ corresponds to the bi-invariant inner product $\langle\rangle$)
  \begin{equation}  \begin{gathered}
\langle\rR^{\ttM}_{\omega_1, \omega_2}\omega_1, \rP_t\omega_2  \rangle =
\frac{1}{4}\|[\omega_{1\frn}, \omega_{2\frn}]_{\frn} + t [\omega_{1\fra}, \omega_{2\frn}] + t[\omega_{1\frn},\omega_{2\fra}]\|^2 +\\
\frac{1}{4}\|[\omega_{1\frn}, \omega_{2\frn}]_{\fra} + t^2[\omega_{1\fra},\omega_{2\fra}]\|^2 
+\frac{1}{4}t(1-t)^3\|[\omega_{1\fra}, \omega_{2\fra}]\|^2 +\\
\frac{3}{4}(1-t)\|[\omega_{1\frn},\omega_{2\frn}]_{\fra} + t[\omega_{1\fra}, \omega_{2\fra}]\|^2 + \frac{3}{4}\|[\omega_1, \omega_2]_{\frk}\|^2
\end{gathered}  
\end{equation}
It is also a weighted sum of squares in a different format from \cref{eq:sec_sum_sq}. It is shown to imply both the non-negativity of curvature when $t\leq 1$ and in the case $\fra$ is abelian, when $t\leq 4/3$. 
\end{remark}
\section{Discussion} In this paper, we have obtained explicit formulas for curvatures of real Stiefel manifolds with deformation metrics and obtained several results related to Einstein metrics and sectional curvature range, including parameter values corresponding to non-negative sectional curvatures. We expect similar results could be obtained for complex and quaternionic Stiefel manifolds. We hope the availability of explicit curvature formulas for a family of metrics on an important class of manifolds will be helpful in both theory and applications. The framework to compute the Levi-Civita connection and curvature for deformations of normal homogeneous spaces could be applied to other families of manifolds, potentially allowing the construction of new Einstein manifolds or manifolds with non-negative curvatures.
\appendix
\section{A few trace formulas} We collect a few results on the trace of common operators that will be useful in the computation of the Ricci curvature for matrix spaces. They are most likely known, but we do not have the exact references.
\begin{lemma} \label{lem:mat_traces}1. Let $X$ be a matrix in $\R^{m\times n}$. The trace of the operator $X\mapsto AXB$ where $A\in\R^{m\times m}$ and $B\in\R^{n\times n}$ is $\Tr(A)\Tr(B)$. In particular, the trace of $X\mapsto AX$ is $n\Tr(A)$, the trace of $X\mapsto XB$ is $m\Tr(B)$. The trace of the operator $X\mapsto AX^{\ft}B$ where $A$ and $B$ are matrices of size $m\times n$ is $\Tr(AB^{\ft})$.\hfill\break
2. The trace of the operator $X\mapsto AXB +B^{\ft}XA^{\ft}$ from the space $\Herm{p}$ to itself is $\Tr(A)\Tr(B)+\Tr(AB^{\ft})$. In particular, the trace of the operator $X\mapsto AX +XA^{\ft}$ is $(p+1)\Tr(A)$. The trace of the operator $X \mapsto \Tr(AX)B$, with $B$ is a symmetric matrix and $A$ is a $p\times p$ matrix is $\Tr(\frac{1}{2}(A+A^{\ft})B)$.\hfill\break
3. The trace of the operator $X\mapsto AXB + B^{\ft}XA^{\ft}$, from the space $\oo(p)$ to itself, where $A$ and $B$ are $p\times p$ matrices, is $\Tr(A)\Tr(B)-\Tr(AB^{\ft})$. In particular, if $A$ and $B$ are antisymmetric matrices then the trace of $X\mapsto[[AX]B]$ is $(2-p)\Tr(AB)$.
\end{lemma}  
\begin{proof} Let $E_{ij}$ be the matrix with the $ij$-entry equal to $1$, and other entries equal to 0 and of the same size as $X$. All the statements are proved similarly, by summing the coefficients of the operators on an appropriate base based on $E_{ij}$. Let entries of $A$ be $a_{ij}$ and entries of $B$ be $b_{ij}$.
  
  For item 1, $(AE_{ij}B)_{ij} = a_{ii}b_{jj}$, so the trace of $X\mapsto AXB$ is $\sum_{ij}a_{ii}b_{jj} = \Tr(A)\Tr(B)$. Since $(AE_{ij}^{\ft}B)_{ij} = a_{ij}b_{ij}$, $\Tr(X\mapsto AX^{\ft}B)$ is $\sum_{ij}a_{ij}b_{ij} = \Tr(AB^{\ft})$.

For item 2, a basis of $\Herm{p}$ consists of matrices $E_{ii}$ $(i=1,\cdots, p)$ and $E_{ij} + E_{ji}$ for $i<j$. We now compute the trace of $X\mapsto AXB +B^{\ft}XA^{\ft}$ with respect to this basis. For $E_{ii}$, $(AE_{ii}B +B^{\ft}E_{ii}A^{\ft})_{ii} = 2a_{ii}b_{ii}$, for $E_{ij} + E_{ji}$, the coefficient is
  $$(A(E_{ij}+ E_{ji})B +B^{\ft}(E_{ij}+ E_{ji})A^{\ft})_{ij}= a_{ii}b_{jj} + a_{ij}b_{ji} +b_{ii}a_{jj} + b_{ij}a_{ij}$$
  Hence the trace is
  $$\sum_i2a_{ii}b_{ii} + \sum_{i<j}(a_{ii}b_{jj} + a_{ij}b_{ij} +b_{ii}a_{jj} + b_{ij}a_{ij})= \sum_ia_{ii}\sum_jb_{jj}+\sum_{ij}a_{ij}b_{ij}$$
which is $\Tr(A)\Tr(B) + \Tr(AB^{\ft})$, as $\sum_ia_{ii}b_{ii} + \sum_{i<j}(a_{ii}b_{jj} + b_{ii}a_{jj})$ rearranges to the first sum, and the sum of remaining terms is $\Tr(AB^{\ft})$.\hfill\break
With $B=\dI_p$ we have the trace of $X\mapsto AX +XA^{\ft}$ is $(p+1)\Tr(A)$.\hfill\break
For the operator $X\mapsto \Tr(AX)B$, the coefficient corresponding to $E_{ii}$ is $a_{ii}b_{ii}$, corresponding to $E_{ij} + E_{ji}$ is $(a_{ij} + a_{ji})b_{ij}$. The trace is
$$\sum_{i} a_{ii}b_{ii} + \sum_{i< j} (a_{ij} + a_{ji})b_{ij} = \frac{1}{2}\Tr((A+A^{\ft})B)$$
  
For item 3, a basis of $\oo(p)$ consist of matrices $E_{ij} - E_{ji}$ for $i<j$. The coefficient corresponds to $E_{ij} - E_{ji}$ is
$$(A(E_{ij} - E_{ji})B +B^{\ft}(E_{ij} - E_{ji})A^{\ft})_{ij}= a_{ii}b_{jj} - a_{ij}b_{ij} +b_{ii}a_{jj} - b_{ji}a_{ji}$$
The trace is $\sum_{i<j}a_{ii}b_{jj} - a_{ij}b_{ij} +b_{ii}a_{jj} - b_{ji}a_{ji} =\sum_{ij}a_{ii}b_{jj} - \sum_{ij}a_{ij}b_{ij}$, which is $\Tr(A)\Tr(B) - \Tr(AB)$.

For the trace of $X\mapsto [[AX]B] = (AX-XA)B - B(AX-XA) = AXB +BXA -BAX-XAB$, we have:
  $$\Tr(X\mapsto AXB + BXA) = -\Tr(AB^{\ft}) = \Tr(AB)$$
  $$\Tr(X\mapsto BAX+XAB) = \Tr(\dI_p)\Tr(BA) - \Tr(BA) =(p-1)\Tr(BA)$$
  from here we get $\Tr(X\mapsto [[AX]B]) = (2-p)\Tr(AB)$.
\end{proof}
\bibliographystyle{amsplain}
\bibliography{RiemannianCurvature}
\end{document}